\documentclass[11pt]{article}
\usepackage{mathrsfs}
\usepackage{amsmath}
\usepackage{color}
\usepackage{bbm}
\usepackage{amsmath,amsthm,amssymb,amscd}
\usepackage{latexsym}
\usepackage{hyperref}
\usepackage[numbers,sort&compress]{natbib}
\usepackage{hypernat}
\usepackage{indentfirst}

\newtheorem{theorem}{Theorem}[section]

\newtheorem{lemma}[theorem]{Lemma}

\newtheorem{definition}[theorem]{Definition}

\newtheorem{remark}[theorem]{Remark}

\newtheorem{proposition}[theorem]{Proposition}

\newtheorem{assumption}[theorem]{ Assumption}
\numberwithin{equation}{section}

\begin{document}
	
	\title{{\bf Regularity of global attractors for beam equations with fractional damping and memory}
		\thanks{The work was supported partly by the NSF of China
			(12171094), the Shanghai Key
			Laboratory for Contemporary Applied Mathematics (08DZ2271900), and Fuyang Normal University (2025KYQD0157).}}
	\author{Yu-Ying Duan $^{a}$,  Ti-Jun Xiao $^{b}$\thanks{Corresponding author. E-mail: tjxiao@fudan.edu.cn (T.J. Xiao).}\\
{\small $^a$ School of Mathematics and Statistics, Fuyang Normal University,}\\
{\small Fuyang 236037, China}\\
		{\small $^b$ Shanghai Key Laboratory for Contemporary Applied Mathematics,}\\
		{\small School of Mathematical Sciences, Fudan University, Shanghai 200433, China}}

    \date{}
	\maketitle

	\begin{abstract}
		This paper investigates the long-time behavior of a semilinear beam equation in a domain $\Omega \subset R^{n}$, with memory and fractional damping of the form $(-\Delta)^{\alpha}u_{t}$ ($\alpha \in [0,2]$ the dissipation index).
Two critical growth indices of the nonlinear term are determined for smooth and $C^2$ boundaries respectively, concerning the existence of the associated semigroup.
We prove the existence of global attractor for the semigroup by showing that it possesses a bounded absorbing set and asymptotic compactness. Furthermore, we find out a new way to obtain, for all $\alpha$, higher regularities than anticipated for the attractors,
and the regularity result indicates an interesting phenomenon
that even much weaker damping can produce regularity that is infinitely close to that in the case of strong damping ($\alpha =2$).
 As a consequence, our regularity result deepens and extends the existing related ones for the case
when the memory is absent.

	\end{abstract}
	\vspace{0.4cm}
	
 \indent Memory.

\vspace{0.4cm}


	\section{Introduction}
	
	Let $\Omega$ be a bounded domain in $R^{n} ~ (n\in N)$ with boundary $\Gamma$ of $C^2$. We are concerned with longtime behaviors of solutions of a semilinear beam equation with fractional damping and memory:
	\begin{equation}\label{1.1}
		\left\{
		\begin{array}{l}
			u_{tt}+\Delta^{2}u+(-\Delta)^{\alpha}u_{t}-\displaystyle\int_{0}^{\infty}g(s)\Delta^{2}  u(t-s)ds+f(u)=h(x), ~(x,t)\in \Omega\times R^{+},\\
			u=\Delta u=0,  ~(x,t)\in \Gamma \times R,\\
			u(x,\tau)=u_{0}(x,\tau), u_{t}(x,0)=u_{1}(x),~ x\in \Omega,~ \tau\le 0.
		\end{array}
		\right.
	\end{equation}
Here $(-\Delta)^{\alpha}u_{t}$ is the fractional damping ($\alpha\in[0,2]$), and the damping with $\alpha\in(0,2)$ is stronger than the weak damping ($\alpha = 0$) yet weaker than the strong damping ($\alpha = 2$); the nonnegative and decreasing function $g$ serves as the memory kernel and $u(t-s)$ reflects the system's past history; $f(u)$ is a nonlinear source, and $h(x)$ a given external force.
In viscoelastic beam materials, hereditary viscoelastic effects may coexist with instantaneous structural dissipation. As a time delayed nonlocal effect, the memory term effectively captures the influence of the material's deformation history, whereas the fractional damping serves as an effective model of frequency dependent structural damping (cf. \cite{Chen-1982}). Thus, the interplay between the two damping mechanisms is relevant for both physical modeling and the long-time behavior of the system.

 Denote by $A$ the negative Laplacian $-\Delta$ in $L^2(\Omega)$, with domain $D(A) = H^2(\Omega)\cap H_{0}^{1}(\Omega)$, and $A^\alpha : D(A^\alpha) \subset L^2(\Omega) \to L^2(\Omega)$ represents the fractional power of $A$ of order $\alpha \in [0,2]$. For simplicity, we write $V_{2s} := D(A^s)$ and endow it with the following inner product and norm
	\begin{equation*}
		\langle u,v\rangle _{V_{s}}=\langle A^{\frac{s}{2}}u, A^{\frac{s}{2}}v\rangle, ~\|u\|_{V_{s}}=\|A^{\frac{s}{2}}u\|,
	\end{equation*}
	where $\|\cdot\|$ and $\langle \cdot, \cdot\rangle$ denote the
	$L^{2}$-norm and $L^{2}$-inner product, respectively. Then $$V_0 = L^2(\Omega), V_1 = H_0^1(\Omega), V_2 = H^2(\Omega)\cap H_{0}^{1}(\Omega), ~\mbox{and}~V_s \hookrightarrow L^{\frac{2n}{n-2s}}(\Omega)~(s \ge 0).$$ Moreover,
	\begin{equation*}\label{V}
		\|u\|_{V_{s}}\le c\|u\|_{V_{r}}, ~s<r,
	\end{equation*}
with some constant $c>0$.
	
	As in \cite{Dafermos-1970}, denoting
	\begin{align*}
		\eta(x,t,s)=u(x,t)-u(x,t-s),
	\end{align*}
	we can transform the nonautonomous \eqref{1.1} into the following autonomous system
	\begin{equation}\label{1.2}
		\left\{
		\begin{array}{l}
			u_{tt}+kA^{2}u+A^{\alpha}u_{t}+\displaystyle\int_{0}^{\infty}g(s)A^{2}\eta(x,t,s)ds+f(u)=h(x),~(x,t)\in \Omega\times R^{+},\\
			\eta_{t}=-\eta_{s}+u_{t},~(x,t,s)\in \Omega\times R^{+}\times R^{+},\\
			u=\Delta u=\eta=\Delta \eta=0,~ (x,t)\in \Gamma \times R^{+}, ~s\in R^{+},\\
			u(x,\tau)=u_{0}(x,\tau), u_{t}(x,0)=u_{1}(x), ~x\in \Omega, ~\tau\le 0,\\
			\eta(x,0,s)=\eta_{0}(x,s), ~(x,s)\in\Omega\times R^{+},
		\end{array}
		\right.
	\end{equation}
	where $k:=1-\int_{0}^{\infty}g(s)ds$. Let us define the history space with respect to memory kernel $g$ by
	\begin{align*}
		\mathcal{M}_{s}=L_{g}^{2}(R^{+}, V_{s})=\left\{\eta\in V_{s} ;~ \|\eta\|_{\mathcal{M}_{s}}<\infty\right\},
	\end{align*}
	with inner and norm given by
	\begin{align*}
		\langle \eta,\zeta\rangle_{\mathcal{M}_{s}}=\int_{0}^{\infty}g(\tau)\langle\eta,\zeta\rangle_{V_{s}}d\tau, ~~\|\eta\|_{\mathcal{M}_{s}}^{2}=\int_{0}^{\infty}g(\tau)\|\eta\|_{V_{s}}^{2}d\tau,~~\eta,\zeta\in \mathcal{M}_{s}.
	\end{align*}
	Moreover, our phase space is
	\begin{align*}
		\mathcal{H}=V_{2}\times V_{0}\times \mathcal{M}_{2},
	\end{align*}
	which is equipped with the norm
	\begin{align*}
		\|(u,u_{t},\eta)\|_{\mathcal{H}}^{2}=k\|\Delta u\|^{2}+\|u_{t}\|^{2}+\|\eta\|_{\mathcal{M}_{2}}^{2}.
	\end{align*}
	Additionally, we define regular (or weak regular) space by
	\begin{align*}
		\mathcal{H}^{s}=V_{s+2}\times V_{s}\times \mathcal{M}_{s+2},
	\end{align*}
	endowed with the norm
	\begin{align*}
		\|(u,u_{t},\eta)\|_{\mathcal{H}^{s}}^{2}=k\|u\|_{V_{s+2}}^{2}+\|u_{t}\|_{V_{s}}^{2}+\|\eta\|_{\mathcal{M}_{s+2}}^{2}.
	\end{align*}
	
		There has been substantial research on the long-time behavior of wave equations with fractional damping; see, e.g., \cite{Azevedo-2023,Li-2021, Yang-2018} and the references therein. Part of the analysis has been extended to the case of extensible beam equations
	\begin{align}\label{1.111}
		u_{tt}+\Delta^{2}u-\kappa M\left(\|\nabla u\|^{2}\right)\Delta u=p(u,u_{t},x)
	\end{align}
	($\kappa\ge0$), which can be viewed as a special case of the Berger plate model \cite{Berger-1955}. Some studies have investigated the existence of global attractors and finite dimensional exponential attractors for  models of type \eqref{1.111} (cf. \cite{Silva-2015,Gomes-2024,Freitas-2025,Ma-2010}). Additionally, research on the regularity and smoothness of global attractors for beam equations with fractional damping has yielded notable results.  For the beam equation with fractional damping $(-\Delta)^{\alpha}u_{t}$, where $\alpha \in (1,2)$ and $\kappa = 0$, Yang and Ding \cite{Yang-2017} identified a critical exponent $p_{\alpha} = \frac{n+2(2\alpha-1)}{(n-2(2\alpha-1))^{+}}$ for Boussinesq type equations with nonlinear term $-\Delta f(u)$ and showed that, for $1 \leq p < p_{\alpha}$, the corresponding dynamical system admits a global attractor, along with an exponential attractor in the strong topology. More recently,  Liu, Yang, and Guo \cite{Liu-2024} established the optimal subcritical growth condition for the nonlinear term $f(u)$ with respect to the dissipation exponent $\alpha$, namely $1 \leq p < p_{\alpha}' = \frac{n+4\alpha}{(n-4)^{+}}$ for $\alpha \in (0,1)$. They proved the existence of strong global attractors and exponential attractors, which possess regularity $V_{4}\times V_{2+\alpha}$. For additional insights regarding the regularity of attractors, one may consult references \cite{Giorgi-2009, la26, Yang-2013}.

	In this paper, we study fractional damping models that incorporate memory effects (arising from viscoelastic materials). Due to the presence of
the memory term, the embedding relation  $\mathcal{M}_{s_{1}} \subset \mathcal{M}_{s_{2}}$ for $s_{1}>s_{2}$, is continuous but not compact, necessitating the use of alternative techniques to prove asymptotic compactness. There are a lot of researches regarding global attractors for wave equations with memory (cf. \cite{Cavalcanti-2016, Conti-2005-, Messaoudi-2017,Pata-2001}). Also, there are some works \cite{Di Plinio-2008,Ma-2004,Yang-2024} on global attractors for wave equations with memory and weak damping $u_t$ or strong damping $-\triangle u_t$.

However, to the best of our knowledge, no regularity results have been established for attractors of beam equations with {\bf memory}. By a regularity result for the global attractor of a semigroup $S(t)$ on a phase space $H$ of functions on $\Omega$, we mean that the attractor lies in a subset of $H$ consisting of more (spatially) regular functions.

This paper aims to fill this gap by performing a regularity analysis for global attractors of \eqref{1.1} and, more significantly, to achieve regularity estimates that exceed what would be predicted by simply analogizing from the known results of wave equations with memory or beam equations with fractional damping. The method developed here may also be applied to wave equations with fractional damping and memory, yielding higher regularity than known results (see Remark \ref{re}).

 The main features and novelties of this paper are as follows.

\begin{enumerate}

\item[\rm{(1)}] This work reveals the distinct effects of memory and fractional damping on attractor regularity and nonlinear critical growth. The memory mechanism governs history and long-time relaxation, whereas the fractional damping provides velocity regularity control. Meanwhile, the fractional damping helps enlarge the admissible growth range of the nonlinearity, while the memory term may restrict this gain. By characterizing the interaction between the two mechanisms, we determine the critical growth threshold of the nonlinearity and establish higher regularity attractor beyond the existing results for related beam equations.

\item[\rm{(2)}]  The presence of the memory term, with a weaker exponential-decay condition than the traditional one for the memory kernel $g$, necessitates the introduction of additional auxiliary
functions, which affect the critical exponents for nonlinear terms under $C^2$ boundaries.
	
		\item[\rm{(3)}] We find out an effective way to achieve {\it high regularities} of global attractors, that is, to derive regularity estimates for the time derivatives through a generalized quasi-stability inequality, and then utilize a bootstrap argument. This enables us to establish, for every $\alpha\in[0,2]$, the regularity of the attractor in the space $V_{4} \times V_{4-\epsilon} \times \mathcal{M}_{4}$ ($\epsilon$ can be an arbitrarily small positive constant, and 0 if $\alpha=2$).
Interestingly, this means that even much weaker damping can
 produce a regularity that is infinitely close to that in the case of strong damping ($\alpha=2$).
  In the special case when the memory term vanishes, i.e., $g(s) \equiv 0$,  the above regularity space becomes $V_{4} \times V_{4-\epsilon}$ (see Remark \ref{re1}), which improves on the regularity space $V_{4} \times V_{2+\alpha}$ given in \cite{Liu-2024} for $\alpha\in(0,1]$. For the case of $\alpha\in(1,2]$, our result also supplements the regularity results of the global attractor for the beam equation; actually, there seems to have been no results yet about attractors with higher regularity than the underlying regularity of the phase spaces (cf. \cite{Yang-2017}).

		\item[\rm{(4)}] The existing literature has studied the regularities of global attractors mainly under smooth boundary conditions (as in \cite{Liu-2024,Yang-2017}). Here, we consider both smooth and $C^{2}$ boundaries and compare the two cases. The $C^2$ case requires a stricter condition on the nonlinear growth index to ensure the continuity of the associated semigroup from $\mathcal{H}$ to itself. Furthermore, the proof strategies differ: for the smooth boundary, the higher regularity of solutions allows us to apply the bootstrap argument more directly, while for the $C^2$ boundary, before applying the bootstrap argument we need to estimate the difference between two trajectories in order to control the regularity of their time derivatives.

	\end{enumerate}

 Some main results in the paper have been stated in \cite{Duan25}. Moreover, in our future work, we will pay attention to the models with nonautonomous damping or fractional brownian motions, since there are many good developments in the study of the relevant systems involving variable coefficients, nonautonomous damping, random disturbances or fractional brownian motions (cf. \cite{IX24,LX20,LX25,PSX24} and references therein).
	
 	The main structure of this paper is as follows. In Section 2, we  provide some preliminary material and the well-posedness of the solution. The existence of the global attractor is established in Section 3.  The higher regularity of the attractor is presented in Section 4.
	
	Throughout this paper, $C(r)$ denotes a positive constant depending on $r$ while $C$ is a generic positive constant, and $\varepsilon, \varepsilon_{0}$ stand for small positive constants, which may vary from line to line.
	\section{Preliminaries and Well-Posedness}
	First, we state the following basic assumptions.
	\begin{assumption}\label{ass1}
		\begin{enumerate}
			\item [\rm{(H1)}] $f\in C^{1}(R)$ with $f(0)=0$, satisfying
			\begin{align}\label{f1}
				\liminf_{|s|\to \infty}\frac{f(s)}{s}>-k\lambda_{1},
			\end{align}
			where $\lambda_{1}>0$ is the principal eigenvalue of operator $\Delta^{2} $ in $H^{2}(\Omega)\cap H_{0}^{1}(\Omega)$, and
			\begin{align}\label{f2}
				|f^{\prime}(s)|\le C\left(1+|s|^{p-1}\right),~ s\in R,
			\end{align}
			with
			\begin{equation*}
				1\le p<\left\{
				\begin{array}{l}
					\infty, ~~1\le n\le 4, \\
					\frac{n+4}{n-4},~ n\ge 5.
				\end{array}
				\right.
			\end{equation*}
			\item[{\rm(H2)}] $g:\lbrack 0,+\infty)\to [0,+\infty)$ is a decreasing and locally absolutely continuous function  which satisfies
			\begin{align}\label{g1}
				G(s)\le \delta g(s),~~s\ge 0,
			\end{align}
			where $G(s):=\int_{s}^{\infty}g(\tau)d\tau$, and $\delta>0$ is a constant.
			\item[{\rm(H3)}] $h(x)\in L^{2}(\Omega)$.
		\end{enumerate}
	\end{assumption}
	\begin{remark}{\rm
		Take $\lambda\in (0,k\lambda_{1})$; the condition \eqref{f1} implies the existence of a positive constant $C$ such that
		\begin{align}\label{f3}
			f(s)s\ge -\lambda s^{2}-C,~ F(s)=\int_{0}^{s}f(\tau)d\tau\ge -\frac{\lambda}{2}s^{2}-C, ~s\in R.
		\end{align}

	The condition \eqref{g1} has been used in several papers (cf. \cite{Di Plinio-2008,Gatti-2008}) and is weaker than the usual exponential-decay condition  $g'(s)+\delta g(s)\le 0$ ($s\ge 0$) assumed in, e.g., \cite{Cavalcanti-2016,Ma-2010,Yang-2024};  the latter condition restricts $g$ from having horizontal line segments, even horizontal inflection points. As shown in \cite{Gatti-2008}, condition \eqref{g1} is equivalent to
		\begin{align}\label{G}
			g(s+\tau)\le K e^{-\delta\tau}g(s),
		\end{align}
		for every $\tau>0, s>0$ and some $K\ge1$.}
	\end{remark}

	\subsection{Preliminary lemmas}
	Let us look at the equations
	\begin{equation}\label{1.3-}
		\left\{
		\begin{array}{l}
			\varphi_{tt}+kA^{2}\varphi+A^{\alpha}\varphi_{t}+\displaystyle\int_{0}^{\infty}g(s)A^{2}\psi(x,t,s)ds+\mathcal{F}(x,\varphi)=0, ~(x,t)\in \Omega\times R^{+},\\
			\psi_{t}=-\psi_{s}+\varphi_{t}, ~ (x, t, s)\in\Omega\times R^{+}\times R^{+}, \\
			\varphi=\psi=\Delta \varphi=\Delta\psi=0, ~(x,t,s)\in \Gamma\times R^{+}\times R^{+};
		\end{array}
		\right.
	\end{equation}
	here, $\psi(x,t,s)=\varphi(x,t)-\varphi(x,t-s)$, and $\mathcal{F}$ is the force function to be determined. For \eqref{1.3-}, we introduce the auxiliary functions:
	\begin{align*}
		E_{\sigma_{1}}(\varphi)=\frac{1}{2}\|(\varphi,\varphi_{t},\psi)\|_{\mathcal{H}^{\sigma_{1}}}^{2},
	\end{align*}
	\begin{align*}
		I_{\sigma_{2}}(\varphi)=\int_{\Omega}\varphi_{t}A^{\sigma_{2}}\varphi dx+\frac{1}{2}\|\varphi\|_{V_{\alpha+\sigma_{2}}}^{2},
	\end{align*}
	\begin{align}\label{P4}
		J_{\sigma_{2}}(\varphi)=\int_{0}^{\infty}G(s)\|\psi-\varphi\|_{V_{\sigma_{2}+2}}^{2}ds.
	\end{align}
	In particular, for $\mathcal{F}(x,\varphi)=f(\varphi)-h(x)$, set
	\begin{align*}
		\mathcal{E}_{\sigma_{1}}(\varphi)=E_{\sigma_{1}}(\varphi)+\int_{\Omega}f(\varphi)A^{\sigma_{1}}\varphi dx-\int_{\Omega}h(x)A^{\sigma_{1}}\varphi dx.
	\end{align*}
	After differentiating the above functions, we obtain the following formulas and estimate:
	\begin{align}\label{P5}
			\frac{d}{dt}E_{\sigma_{1}}(\varphi)+\|\varphi_{t}\|_{V_{\sigma_{1}+\alpha}}^{2}-\frac{1}{2}\int_{0}^{\infty}g^{\prime}(s)
\|\psi\|_{V_{\sigma_{1}+2}}^{2}ds=-\int_{\Omega}\mathcal{F}(x,\varphi)A^{\sigma_{1}}\varphi_{t}dx,
		\end{align}
				\begin{align}\label{P7}
			&\frac{d}{dt}I_{\sigma_{2}}(\varphi)+k\|\varphi\|_{V_{2+\sigma_{2}}}^{2}-\|\varphi_{t}\|_{V_{\sigma_{2}}}^{2}\nonumber\\
			=&-\displaystyle\int_{0}^{\infty}g(s)\int_{\Omega}A^{\frac{2+\sigma_{2}}{2}}\psi A^{\frac{2+\sigma_{2}}{2}}\varphi dxds-\int_{\Omega}\mathcal{F}A^{\sigma_{2}}\varphi dx,
		\end{align}
		\begin{align}\label{P8}
			\frac{d}{dt}J_{\sigma_{2}}(\varphi)+\|\psi\|_{\mathcal{M}_{\sigma_{2}+2}}^{2}=2\displaystyle\int_{0}^{\infty}g(s)\int_{\Omega}A^{\frac{2
+\sigma_{2}}{2}}\psi A^{\frac{2+\sigma_{2}}{2}}\varphi dxds,
		\end{align}
		\begin{align}\label{P6}
			\frac{d}{dt}\mathcal{E}_{\sigma_{1}}(\varphi)+\|\varphi_{t}\|_{V_{\sigma_{1}+\alpha}}^{2}-\frac{1}{2}\int_{0}^{\infty}
g^{\prime}(s)\|\psi\|_{V_{\sigma_{1}+2}}^{2}ds=\int_{\Omega}f^{\prime}(\varphi)\varphi_{t}A^{\sigma_{1}}\varphi dx,
		\end{align}
		and for $\varepsilon$ small enough and $\sigma_{1}\le\sigma_{2}\le \sigma_{1}+\alpha$,
		\begin{align}\label{P9}
			C E_{\sigma_{1}}(\varphi)\le   E_{\sigma_{1}}(\varphi)+\varepsilon I_{\sigma_{2}}(\varphi)+\frac{\varepsilon}{2}J_{\sigma_{2}}(\varphi)\le C\left(E_{\sigma_{1}}(\varphi)+\|\psi\|_{\mathcal{M}_{2+\sigma_{2}}}^{2}+\|\varphi\|_{V_{2+\sigma_{2}}}^{2}\right).
		\end{align}

	\subsection{Generation of a Dynamical System}
	First, we give the existence result for weak solutions.
\begin{proposition}\label{th1}
		Let the Assumption \ref{ass1} hold. Then for initial datum  $$U_{0}:=(u_{0},u_{1},\eta_{0})\in \mathcal{H},$$ the problem \eqref{1.2} have a weak solution $U:=(u,u_{t}, \eta)\in C([0,T];\mathcal{H})$.
	\end{proposition}
	\begin{proof} We employ the standard Faedo-Galerkin method (cf. \cite{Li-2020,Pata-2001}).
		\begin{enumerate}
			\item[\rm (1)] Let $\{\omega_{i}\}_{i\ge 1}$ be an orthogonal basis for $H_{0}^{1}(\Omega)$ which is orthonormal in $L^{2}(\Omega)$ with
			\begin{equation*}
				\left\{
				\begin{array}{l}
					-\Delta \omega_{i}=\lambda_{i}\omega_{i}, ~ x\in\Omega,\\
					~~~~~\omega_{i}=0, ~~~~~ x\in \Gamma.
				\end{array}
				\right.
			\end{equation*}
			By the ODE theory, we can find $\{a_{im}(t)\}_{i\ge 1}$ and $\{b_{im}(t-s)\}_{i\ge 1}$ for $t,s\in R^{+}$ such that
			\begin{align*}
				u_{m}=\sum_{i=1}^{m}a_{im}(t)\omega_{i}(x),~~\eta_{m}(x,t,s)=\sum_{i=1}^{m}\left(a_{im}(t)-b_{im}(t-s)\right)\omega_{i}(x),
			\end{align*}
			which is the solution to the following approximate problems
			\begin{equation}\label{1.3}
				\left\{
				\begin{array}{l}
					\left(u_{mtt}+kA^{2}u_{m}+\displaystyle\int_{0}^{\infty}g(s)A^{2}\eta_{m}(s)ds, \omega_{i}\right)+\left(f(u_{m})+A^{\alpha}u_{mt},\omega_{i}\right)=(h,\omega_{i}), \\
					(u_{m}(0), \omega_{i})=\varphi_{im}, ~  (u_{mt}(0), \omega_{i})=\psi_{im}, ~(\eta_{m}(0,s),\omega_{i})=\xi_{im}, i=1,2,...m.
				\end{array}
				\right.
			\end{equation}
			Here $\varphi_{im}, \psi_{im}, \xi_{im}$ are chosen such that
			\begin{equation*}
				\left\{
				\begin{array}{l}
					\sum_{i=1}^{m}\varphi_{im}\omega_{i}\to u_{0}\mbox{~in~}V_{2},\\
					\sum_{i=1}^{m}\psi_{im}\omega_{i}\to u_{1}\mbox{~in~}L^{2}(\Omega),\\
					\sum_{i=1}^{m}\xi_{im}\omega_{i}\to \eta_{0}\mbox{~in~} \mathcal{M}_{2}.
				\end{array}
				\right.
			\end{equation*}
			Multiplying equation \eqref{1.3} by $a_{im}^{\prime}(t)$, then summing with respect to $i$, we get
			\begin{align*}
				&\frac{d}{dt}\left( \frac{1}{2}\|u_{mt}\|^{2}+\frac{k}{2}\|Au_{m}\|^{2}+\frac{1}{2}\|\eta_{m}\|_{\mathcal{M}_{2}}^{2}+\int_{\Omega}F(u_{m})dx\right)\nonumber\\
				&+\|A^{\frac{\alpha}{2}}u_{mt}\|^{2}-\frac{1}{2}\int_{0}^{\infty}g^{\prime}(s)\|\eta_{m}^{t}(s)\|_{V_{2}}^{2}ds=(h(x), u_{mt}).
			\end{align*}
			By the assumption \ref{ass1}, and integrating with respect to $t$, we obtain
			\begin{align}\label{1.4}
				\|u_{mt}\|^{2}+k\|Au_{m}\|^{2}+\|\eta_{m}\|_{\mathcal{M}_{2}}^{2}+2\int_{\Omega}F(u_{m})dx+\int_{0}^{T} \|A^{\frac{\alpha}{2}}u_{mt}\|^{2}dt\le C_{T}.
			\end{align}
			Estimate \eqref{1.4} means
			\begin{equation*}
				\left\{
				\begin{array}{l}
					\{u_{m}\}\mbox{~ uniformly ~bounded ~in~}L^{\infty}\left([0,T]; V_{2}\right),\\
					\{u_{mt}\}\mbox{~ uniformly ~bounded ~in~}L^{\infty}([0,T]; L^{2}(\Omega))\cap L^{2}\left([0,T];V_{\alpha}\right),\\
					\{\eta_{m}\}\mbox{~ uniformly ~bounded ~in~}L^{\infty}([0,T]; \mathcal{M}_{2}).
				\end{array}
				\right.
			\end{equation*}
			By the Banach-Alaoglu theorem, there exists a function $U:=(u,u_{t},\eta)$ such that
			\begin{equation*}
				\left\{
				\begin{array}{l}
					\{u_{m}\}\to u\mbox{~ weakly ~star ~in~}L^{\infty}([0,T]; V_{2}),\\
					\{u_{mt}\}\to u_{t}\mbox{~ weakly~star ~in~}L^{\infty}([0,T]; L^{2}(\Omega)),\\
					\{A^{\alpha}u_{mt}\}\to A^{\alpha}u_{t}\mbox{~ weakly ~in~}L^{2}([0,T]; (V_{\alpha})^{'}),\\
					\{\eta_{m}\}\to \eta \mbox{~ weakly ~star ~in~}L^{\infty}([0,T]; \mathcal{M}_{2}).
				\end{array}
				\right.
			\end{equation*}
			Applying the Aubin-Lions lemma, we deduce
			\begin{align*}
				u_{m}\to u\mbox{~strongly ~in~} L^{\infty}([0,T]; V_{r}),~ 0\le r<2,
			\end{align*}
			and hence
			\begin{align*}
				u_{m}\to u\mbox{~~a.e. in~} \Omega\times [0,T],
			\end{align*}
			which implies
			\begin{align*}
				f(u_{m})\to f(u)\mbox{~~a.e. in~} \Omega\times [0,T].
			\end{align*}
			By \eqref{f2}, $f(u_{m})$ is uniformly bounded in $L^{\frac{p+1}{p}}(\Omega\times[0,T])$. Therefore, from the Strauss Lemma it follows that
			\begin{align*}
				f(u_{m})\to f(u) \mbox{~weakly ~in~} L^{\frac{p+1}{p}}([0,T];L^{\frac{p+1}{p}}(\Omega)).
			\end{align*}
			Let $\langle\!\langle \cdot , \cdot \rangle\!\rangle$ denote the duality pairing between
			$H_{g}^{1}(\mathbb{R}^{+}, V_{2})$ and its dual space. Similarly as in \cite{131},
			we obtain
			\[
			\lim_{m \to \infty} \langle\!\langle \eta_{m}', \xi \rangle\!\rangle
			= \langle\!\langle \eta', \xi \rangle\!\rangle .
			\]

Next, we show $U$ satisfies the initial condition. Choosing the test function $$(\varphi(t),\xi(t,s))\in C^{\infty}([0,T]; V_{2})\times C^{\infty}([0,T]; C^{\infty}(R^{+}, V_{2}))$$ with $\varphi(T)=0, \xi(T)=0$ in \eqref{1.3} and integrating by parts, we have
			\begin{align*}
				& -\int_{0}^{T} \int_{\Omega}u_{mt}\varphi_{t}dx dt  +k\int_{0}^{T}\int_{\Omega}Au_{m}A\varphi dxdt+\int_{0}^{T}\displaystyle\int_{0}^{\infty}g(s)\int_{\Omega}A\eta_{m}A\varphi dxdsdt\nonumber\\
				& +\int_{0}^{T}\int_{\Omega}(f(u_{m})-h) \varphi dxdt  -\int_{0}^{T} \int_{\Omega}A^{\frac{\alpha}{2}}u_{m}A^{\frac{\alpha}{2}}\varphi_{t}dx\nonumber\\
				=&  \int_{\Omega}u_{mt}(0)\varphi(0)dx +\int_{\Omega}  A^{\frac{\alpha}{2}}u_{m} (0)A^{\frac{\alpha}{2}}\varphi(0) dx,
			\end{align*}
			and
			\begin{align*}
				& - \int_{0}^{T}\displaystyle\int_{0}^{\infty}g(s)\int_{\Omega}\eta_{m}\xi_{t} dxdsdt+\int_{0}^{T}\displaystyle\int_{0}^{\infty}g(s)\int_{\Omega}\eta_{ms}\xi dxdsdt\nonumber\\
				&+\int_{0}^{T}\displaystyle\int_{0}^{\infty}g(s)\int_{\Omega}u_{m}\xi_{t} dxdsdt\nonumber\\
				=&\displaystyle\int_{0}^{\infty}g(s)\int_{\Omega}\eta_{m}(0)\xi(0) dxds-\int_{0}^{T}\displaystyle\int_{0}^{\infty}g(s)\int_{\Omega}u_{m}(0)\xi(0) dxdsdt.
			\end{align*}
			Taking limits as $m\to \infty$, we get
			\begin{align*}
				& -\int_{0}^{T} \int_{\Omega}u_{t}\varphi_{t}dx dt  +k\int_{0}^{T}\int_{\Omega}AuA\varphi dxdt+\int_{0}^{T}\displaystyle\int_{0}^{\infty}g(s)\int_{\Omega}A\eta A\varphi dxdsdt\nonumber\\
				& +\int_{0}^{T}\int_{\Omega}(f(u)-h) \varphi dxdt  -\int_{0}^{T} \int_{\Omega}A^{\frac{\alpha}{2}}u A^{\frac{\alpha}{2}}\varphi_{t}dx\nonumber\\
				=&  \int_{\Omega}u_{t}(0)\varphi(0)dx +\int_{\Omega}  A^{\frac{\alpha}{2}}u (0)A^{\frac{\alpha}{2}}\varphi(0) dx,
			\end{align*}
			and
			\begin{align*}
				& - \int_{0}^{T}\displaystyle\int_{0}^{\infty}g(s)\int_{\Omega}\eta\xi_{t} dxdsdt+\int_{0}^{T}\displaystyle\int_{0}^{\infty}g(s)\int_{\Omega}\eta_{s}\xi dxdsdt\nonumber\\
				&+\int_{0}^{T}\displaystyle\int_{0}^{\infty}g(s)\int_{\Omega}u\xi_{t} dxdsdt\nonumber\\
				=&\displaystyle\int_{0}^{\infty}g(s)\int_{\Omega}\eta(0)\xi(0) dxds-\int_{0}^{T}\displaystyle\int_{0}^{\infty}g(s)\int_{\Omega}u(0)\xi(0) dxdsdt.
			\end{align*}
			This implies that $(u_{m}(0), u_{mt}(0), \eta_{m}(0))\to (u_{0}, u_{1}, \eta_{0})$. Therefore,  $(u,u_{t}, \eta^{t})$ is a weak solution to the problem \eqref{1.2}.
			\item[{\rm(2)}] Following the related arguments in \cite{Li-2020}, one obtains $(u,u_{t}) \in C([0,T]; V_{2} \times L^{2})$, and an argument similar to that in \cite{Pata-2001} gives $\eta \in C([0,T]; \mathcal{M}_{2})$. Consequently, we conclude
			$U \in C([0,T]; \mathcal{H})$.
		\end{enumerate}
	\end{proof}
	
Proposition \ref{th1} allows us to define an operator family $S^{\alpha}(t): \mathcal{H}\to \mathcal{H}$ associated to problem \eqref{1.2} by
		\begin{align*}
			S^{\alpha}(t): U_{0} \to U(t).
		\end{align*}
		Here, $U(\cdot)$ denotes the weak solution to the problem \eqref{1.2} with initial value $ U_0 \in \mathcal{H}$.
	
	Given two initial values $z^{i}=(u_{0}^{i}, u_{1}^{i}, \eta_{0}^{i})\in \mathcal{H}$, where $i=1,2$, the corresponding two solutions to problem \eqref{1.2} are given by
	\begin{align*}
		S^{\alpha}(t)z^{i}=(u^{i},u_{t}^{i},\eta^{i}),~i=1,2.
	\end{align*}
	Denote $z=S^{\alpha}(t)z^{1}-S^{\alpha}(t)z^{2}=(\widehat{u},\widehat{u}_{t},\widehat{\eta})$, that satisfies
	\begin{equation}\label{5.1}
		\left\{
		\begin{array}{l}
			\widehat{u}_{tt}+kA^{2}\widehat{u}+A^{\alpha}\widehat{u}_{t}+\displaystyle\int_{0}^{\infty}g(s)A^{2}\widehat{\eta} ds+f(u^{1})-f(u^{2})=0,~(x,t)\in \Omega\times R^{+},\\
			\widehat{\eta}_{t}=-\widehat{\eta}_{s}+\widehat{u}_{t},~(x,t,s)\in \Omega\times R^{+}\times R^{+}, \\
			\widehat{u}=\widehat{\eta}=\Delta \widehat{u}=\Delta \widehat{\eta}=0,~ (x,t,s)\in \Gamma\times R^{+}\times R^{+},\\
			(\widehat{u},\widehat{u}_{t},\widehat{\eta})|_{t=0}=	z^{1}-z^{2}, ~x\in \Omega, s\in R^{+}.
		\end{array}
		\right.
	\end{equation}
	
	\begin{theorem}\label{thm1}
		Let the Assumption \ref{ass1} hold and let the growth exponent $p$ (in \eqref{f2}) satisfy
		\begin{equation}\label{assp}
			1\le p<\left\{
			\begin{array}{lll}
				\infty, &n=1,2,3,4, \\
				p^{*}:=\min\left\{\frac{n+4\alpha}{n-4}, \frac{n+4}{n-4}\right\}, & n\ge 5.
			\end{array}
			\right.
		\end{equation}
		Then, the solution $U(t)=S^{\alpha}(t)U_{0}$ is unique. This implies that $\{S^{\alpha}(t)\}_{t\ge 0}$ (for each $\alpha\in [0,2]$) is a semigroup on $\mathcal{H}$.

Moreover, if $\Gamma$ is smooth, the solution has the following estimate:
		\begin{align} \label{esofu}
			\|S^{\alpha}(t)U_{0}\|_{\mathcal{H}^{\alpha}}^{2}\le C,
		\end{align}
		and the mapping $S^{\alpha}(t)$ is continuous from $ \mathcal{H}$ to $\mathcal{H}$, for any $t>0$.  If $\Gamma\in C^{2}$ only, we further assume
		\begin{equation}\label{assp1}
			1\le p<\left\{
			\begin{array}{lll}
				\infty, &n=1,2,3,4, \\[0.28cm]
				p^{**}:=\frac{n+2\alpha}{n-4}, &n\ge 5,
			\end{array}
			\right.
	\end{equation}
		then, the mapping $S^{\alpha}(t)$ is continuous from $ \mathcal{H}$ to $\mathcal{H}$, for any $t>0$.
	\end{theorem}
	\begin{proof}
		\begin{enumerate}
			\item[\rm{(1)}]
			First, we let \eqref{assp} be satisfied and $z_{1}=z_{2}$ in \eqref{5.1}, take $\sigma_{1}=0$ and $\sigma_{2}=\alpha$, substitute $\varphi$ with $\widehat{u}$, and set $\mathcal{F}=f(u^{1})-f(u^{2})$ in \eqref{P5}--\eqref{P8}. Then, by \eqref{f2} and the interpolation inequality, we obtain, for any $0<\varepsilon_0<<1$,
			\begin{align*}
				&\frac{d}{dt}\left(E_{0}(\widehat{u})+\varepsilon I_{\alpha}(\widehat{u})+\frac{\varepsilon}{2}J_{\alpha}(\widehat{u})\right)+(1-\varepsilon)\|\widehat{u}_{t}\|_{V_{\alpha}}^{2}+k\varepsilon\|\widehat{u}\|_{V_{2+\alpha}}^{2}+\frac{\varepsilon}{2}\|\widehat{\eta}\|_{\mathcal{M}_{2+\alpha}}^{2}\nonumber\\
				=&\frac{1}{2}\int_{0}^{\infty}g^{\prime}(s)\|\widehat{\eta}\|_{V_{2}}^{2}ds-\int_{\Omega}\left(f(u^{1})-f(u^{2})\right)\left(\widehat{u}_{t}+\varepsilon A^{\alpha}\widehat{u}\right)dx\nonumber\\
				\le& C\int_{\Omega}\left(1+|u_{1}|^{p-1}+|u_{2}|^{p-1}\right)|\widehat{u}|\left(|\widehat{u}_{t}|+\varepsilon|A^{\alpha}\widehat{u}|\right)dx\nonumber\\
				\le& C\|\widehat{u}_{t}\|_{L^{\frac{2n}{n-2\alpha}}}\|\widehat{u}\|_{L^{\frac{2n}{n-2(2+\alpha-\varepsilon_{0})}}}\left(1+\|u_{1}\|_{L^{\frac{(p-1)n}{2+2\alpha-\varepsilon_{0}}}}^{p-1}+\|u_{2}\|_{L^{\frac{(p-1)n}{2+2\alpha-\varepsilon_{0}}}}^{p-1}\right)\nonumber\\
				&+C\varepsilon\|A^{\alpha}\widehat{u}\|_{L^{\frac{2n}{n-2(2-\alpha)}}}\|\widehat{u}\|_{L^{\frac{2n}{n-2(2+\alpha-\varepsilon_{0})}}}\left(1+\|u_{1}\|_{L^{\frac{(p-1)n}{4-\varepsilon_{0}}}}^{p-1}+\|u_{2}\|_{L^{\frac{(p-1)n}{4-\varepsilon_{0}}}}^{p-1}\right)\nonumber\\
				\le&C  \left(\|\widehat{u}_{t}\|_{V_{\alpha}}+\varepsilon\|\widehat{u}\|_{V_{2+\alpha}}\right)\|\widehat{u}\|_{V_{2+\alpha}}^{\frac{2+\alpha-\varepsilon_{0}}{2+\alpha}}\|\widehat{u}\|^{\frac{\varepsilon_{0}}{2+\alpha}}\nonumber\\
				\le& \frac{1-\varepsilon}{2} \|\widehat{u}_{t}\|_{V_{\alpha}}^{2}+\frac{k\varepsilon}{2}\|\widehat{u}\|_{V_{2+\alpha}}^{2}+C(\varepsilon)\|\widehat{u}\|^{2}.
			\end{align*}
			Thus,
			\begin{align}\label{1.9}
				&\frac{d}{dt}\left(E_{0}(\widehat{u})+\varepsilon I_{\alpha}(\widehat{u})+\frac{\varepsilon}{2}J_{\alpha}(\widehat{u})\right)+\frac{1-\varepsilon}{2}\|\widehat{u}_{t}\|_{V_{\alpha}}^{2}\nonumber\\
				&+\frac{k\varepsilon}{2}\|\widehat{u}\|_{V_{2+\alpha}}^{2}+\frac{\varepsilon}{2}\|\widehat{\eta}\|_{\mathcal{M}_{2+\alpha}}^{2}\le C(\varepsilon)\|\widehat{u}\|^{2}.
			\end{align}
			Taking $\varepsilon$ small enough in \eqref{1.9} and \eqref{P9}, we justify
			the uniqueness of the solution by the Gronwall inequality.
			
\item[\rm{(2)}] When $\Gamma \in C^{\infty}$, the elliptic regularity theory indicates that $\omega_i \in C^{\infty}$. Therefore, for any $0 < \mu_{0} < +\infty$, the following holds:
			\begin{align}\label{d}
				\|(u_m(0), u_{mt}(0), \eta_m(0))\|_{\mathcal{H}^{\mu_{0}}} \le C.
			\end{align}
			For the equation \eqref{1.3}, using \eqref{P5} -- \eqref{P8} and taking $\mathcal{F} = f(u_m) - h(x)$, $\sigma_1 = \alpha$ and $\sigma_2 = \min\{2\alpha, 2\}$,  we obtain, for $0 < \varepsilon_0 \ll 1$,
			\begin{align*}
				& \frac{d}{dt}\left( E_{\alpha}(u_m) + \varepsilon I_{\sigma_2}(u_m) + \frac{\varepsilon}{2} J_{\sigma_2}(u_m)\right) + \|u_{mt}\|_{V_{2\alpha}}^2 - \varepsilon \|u_{mt}\|_{V_{\sigma_2}}^2 \\
				& \quad + k\varepsilon \|u_m\|_{V_{2+\sigma_2}}^2 + \frac{\varepsilon}{2} \|\eta_m\|_{\mathcal{M}_{2+\sigma_2}}^2 \\
				\le & - \int_{\Omega} f(u_m) \left(A^{\alpha} u_{mt} + \varepsilon A^{\sigma_2} u_m\right) dx + \int_{\Omega} h(x) \left(A^{\alpha} u_{mt} + \varepsilon A^{\sigma_2} u_m\right) dx \\
				\le & \frac{1}{4} \|u_{mt}\|_{V_{2\alpha}}^2 + \frac{k\varepsilon}{4} \|u_m\|_{V_{2+\sigma_2}}^2 + C\|h\|^2 \\
				& + C\left(1 + \|u_m\|_{L^{\frac{(p-1)n}{2+\sigma_2-\varepsilon_0}}}^{p-1}\right) \|u_m\|_{L^{\frac{2n}{n-2(2+\sigma_2-\varepsilon_0)}}} \left(\|A^{\alpha} u_{mt}\| + \varepsilon \|A^{\sigma_2} u_m\|\right) \\
				\le & \frac{1}{2} \|u_{mt}\|_{V_{2\alpha}}^2 + \frac{k\varepsilon}{2} \|u_m\|_{V_{2+\sigma_2}}^2 + C(\|h\|^2 + 1).
			\end{align*}
			Therefore, letting $\varepsilon$ sufficiently small, we have
			\begin{align*}
				& \frac{d}{dt}\left( E_{\alpha}(u_m) + \varepsilon I_{\sigma_2}(u_m) + \frac{\varepsilon}{2} J_{\sigma_2}(u_m) \right) + \frac{1}{4} \|u_{mt}\|_{V_{2\alpha}}^2 \\
				&+ \frac{k\varepsilon}{2} \|u_m\|_{V_{2+\sigma_2}}^2 + \frac{\varepsilon}{2} \|\eta_m\|_{\mathcal{M}_{2+\sigma_2}}^2\le C(\|h\|^2 + 1).
			\end{align*}
			Using \eqref{P9} and the Gronwall inequality, along with \eqref{d}, gives
			\begin{align*}
				\|(u_m, u_{mt}, \eta_m)\|_{\mathcal{H}^{\alpha}}^2 \le C,
			\end{align*}
			and by the lower semicontinuity of the norm with respect to weak limits, we verify \eqref{esofu}. Next, for the equation \eqref{5.1}, utilizing \eqref{P5}--\eqref{P8} and choosing $\sigma_1 = -\alpha$ and $\sigma_2 = 0$, we obtain
			\begin{align*}
				&\frac{d}{dt}\left(E_{-\alpha}(\widehat{u}) + \varepsilon I_{0}(\widehat{u}) + \frac{\varepsilon}{2}J_{0}(\widehat{u})\right) + (1-\varepsilon)\|\widehat{u}_t\|^2 + k\varepsilon \|\widehat{u}\|_{V_2}^2 + \frac{\varepsilon}{2}\|\widehat{\eta}\|_{\mathcal{M}_2}^2 \nonumber \\
				\leq &\int_{\Omega} (f(u^1) - f(u^2))(A^{-\alpha} \widehat{u}_t + \varepsilon \widehat{u}) \, dx \nonumber \\
				\leq& C \|\widehat{u}\|_{L^{\frac{2n}{n - 2(2 - \varepsilon_0)}}} \left( \|A^{-\alpha} \widehat{u}_t\|_{L^{\frac{2n}{n - 4\alpha}}} + \varepsilon \|\widehat{u}\|_{L^{\frac{2n}{n - 4}}} \right) \nonumber \\
				&\times \left( 1 + \|u^1\|_{L^{\frac{(p-1)n}{2 + 2\min\{\alpha,1\}- \varepsilon_0}}} + \|u^2\|_{L^{\frac{(p-1)n}{2 + 2\min\{\alpha,1\} - \varepsilon_0}}} \right) \nonumber \\
				\leq& \frac{1 - \varepsilon}{2} \|\widehat{u}_t\|^2 + \frac{k\varepsilon}{2} \|\widehat{u}\|_{V_2}^2 + C \|\widehat{u}\|^2.
			\end{align*}
Hence,		
			\begin{align*}
				\|S^{\alpha}(t)z^{1} - S^{\alpha}(t)z^2\|_{\mathcal{H}^{-\alpha}}^2 \leq C(t) \left( \|z^1 - z^2\|_{\mathcal{H}^{-\alpha}}^2 + \|u_0^1 - u_0^2\|_{V_2}^2 + \|\eta_0^1 - \eta_0^2\|_{\mathcal{M}_2}^2 \right),
			\end{align*}
by applying  Gronwall's inequality and \eqref{P9}. Accordingly,
		utilizing the interpolation inequalities and the estimate \eqref{esofu}, we derive
			\begin{align*}
				\|S^{\alpha}(t)z^1 - S^{\alpha}(t)z^2\|_{\mathcal{H}} \leq & C \|S^{\alpha}(t)z^1 - S^{\alpha}(t)z^2\|_{\mathcal{H}^\alpha}^{\frac{1}{2}} \|S^{\alpha}(t)z^1 - S^{\alpha}(t)z^2\|_{\mathcal{H}^{-\alpha}}^{\frac{1}{2}} \\
				\leq & C(t) \left( \|z^1 - z^2\|_{\mathcal{H}^{-\alpha}}^2 + \|u_0^1 - u_0^2\|_{V_2}^2 + \|\eta_0^1 - \eta_0^2\|_{\mathcal{M}_2}^2 \right)^{\frac{1}{4}} \\
				\leq & C(t) \|z^1 - z^2\|_{\mathcal{H}}^{\frac{1}{2}}.
			\end{align*}
			This shows that the mapping $S^{\alpha}(t)$ is continuous from $\mathcal{H}$ to $\mathcal{H}$.
			
			\item [\rm{(3)}] When $\Gamma \in C^{2}$ only, the elliptic regularity theory indicates $\omega_i \in H^2(\Omega)$, which indicates the boundedness of the solution (only) in $\mathcal{H}$. Therefore, letting \eqref{assp1} hold and
			taking $\sigma_{1}= \sigma_{2}=0$ in \eqref{P5}--\eqref{P8} for the equation \eqref{5.1}, we also obtain
			\begin{align*}
				&\frac{d}{dt}\left(E_{0}(\widehat{u})+\varepsilon I_{0}(\widehat{u})+\frac{\varepsilon}{2}J_{0}(\widehat{u})\right)+\|\widehat{u}_{t}\|_{V_{\alpha}}^{2}-\varepsilon\|\widehat{u}_{t}\|^{2}+k\varepsilon\|\widehat{u}\|_{V_{2}}^{2}+\frac{\varepsilon}{2}\|\widehat{\eta}\|_{\mathcal{M}_{2}}^{2}\nonumber\\
				\le& \int_{\Omega}(f(u^{1})-f(u^{2}))(\widehat{u}_{t}+\varepsilon \widehat{u})dx\nonumber\\
				\le&C\|\widehat{u}\|_{L^{\frac{2n}{n-2(2-\varepsilon_{0})}}}\left(\|\widehat{u}_{t}\|_{L^{\frac{2n}{n-2\alpha}}}+\varepsilon\|\widehat{u}\|_{L^{\frac{2n}{n-4}}}\right)\left(1+\|u^{1}\|_{L^{\frac{(p-1)n}{2+\alpha-\varepsilon_{0}}}}+\|u^{2}\|_{L^{\frac{(p-1)n}{2+\alpha-\varepsilon_{0}}}}\right)\nonumber\\
				\le& \frac{1}{2}\|\widehat{u}_{t}\|_{V_{\alpha}}^{2}+\frac{k\varepsilon}{2}\|\widehat{u}\|_{V_{2}}^{2}+C\|\widehat{u}\|^{2}.
			\end{align*}
			Then, letting $\varepsilon$ small enough, from \eqref{P9} we deduce
			\begin{align*}
				\|S^{\alpha}(t)z^{1}-S^{\alpha}(t)z^{2}\|_{\mathcal{H}}\le C(t)\|z^{1}-z^{2}\|_{\mathcal{H}},
			\end{align*}
			which means the continuity of $S^{\alpha}(t)$ on $\mathcal{H}$, for any $t>0$.
			
		\end{enumerate}
	\end{proof}
	\begin{remark}\label{re2.8}{\rm
	\begin{enumerate}
	\item[\rm(i)] It is easy to see that $p^{**}= p^*$ for the cases of weak damping ($\alpha = 0$) and strong damping ($\alpha = 2$).
		\item[\rm(ii)] When the memory term vanishes, i.e., $g\equiv0$, the continuity of the semigroup $S^{\alpha}(t)$ on $\mathcal{H}$ (for any $t>0$) is still guaranteed in the absence of both conditions \eqref{assp1} and $\Gamma\in C^{\infty}$. Indeed,
one can take $\sigma_1 = 0$ and $\sigma_{2}= \min\left\{\alpha,2-\alpha\right\}$ in item (3) of the above proof.
		\end{enumerate}}
	\end{remark}

	\section{Global attractor}
	\subsection{Existence of absorbing set}
	\begin{lemma}\label{lem4.1}
		Let the Assumption \ref{ass1} hold. Then, there exists a bounded absorbing set $\mathcal{B}_{\alpha} = B_{\mathcal{H}}(R_{0})$ in $\mathcal{H}$ for the semigroup $S^{\alpha}(t)$, where $B_{\mathcal{H}}(R_{0})$ denotes a ball of $\mathcal{H}$ centered at $0$ of radius $R_{0}$ (being independent of $\alpha$).
	\end{lemma}
	\begin{proof}
		Replacing $\varphi$ with $u$, taking $\mathcal{F}=f(u)-h(x)$ and $\sigma_{1}=\sigma_2=0$, and applying \eqref{P7}--\eqref{P6},  from \eqref{f3} we deduce, for any $\varepsilon_{0}>0$,
		\begin{align*}
			& \frac{d}{dt}\left(\mathcal{E}_{0}(u)+\varepsilon I_{0}(u)+\frac{\varepsilon}{2}J_{0}(u)\right)+\|u_{t}\|_{V_{\alpha}}+\frac{\varepsilon}{2}\|\eta\|_{\mathcal{M}_{2}}+k\varepsilon\|u\|_{V_{2}}^{2}-\varepsilon\|u_{t}\|^{2}\nonumber\\
			\le&-\varepsilon\int_{\Omega}(f(u)-h(x))udx\le \left(\lambda+\varepsilon_{0}\right)\varepsilon\|u\|^{2}+C\left(\|h\|^{2}+1\right).
		\end{align*}
		 Taking $\varepsilon, \varepsilon_{0}$ small enough, by the Poincar\'e inequality we get
		\begin{align}\label{3.1}
			\frac{d}{dt}\left(\mathcal{E}_{0}(u)+\varepsilon I_{0}(u)+\frac{\varepsilon}{2}J_{0}(u)\right)+C\varepsilon\|(u,u_{t},\eta)\|_{\mathcal{H}}^{2}\le C\left(1+\|h\|^{2}\right),
		\end{align}
		and
		\begin{align}\label{3.2}
			C\|(u,u_{t},\eta)\|_{\mathcal{H}}^{2}-C\le \mathcal{E}_{0}(u)+\varepsilon I_{0}(u)+\frac{\varepsilon}{2}J_{0}(u) \le C\|(u,u_{t},\eta)\|_{\mathcal{H}}^{2}+C.
		\end{align}
		From \eqref{3.1} and \eqref{3.2}, by the Gronwall inequality we deduce, for $\gamma_{0}, R_{0}>0$,
		\begin{align*}
			\|(u,u_{t}, \eta)\|_{\mathcal{H}}^{2}\le C e^{-\gamma_{0} t}\left( \|(u_{0},u_{1}, \eta_{0})\|_{\mathcal{H}}^{2}+1\right)+C\left(1+\|h\|^{2}\right)\le R_{0}^{2}.
		\end{align*}
		Consequently, this implies the existence of a bounded absorbing set $ \mathcal{B}_{\alpha} = B_{\mathcal{H}}(R_0) \subset \mathcal{H}$ for $S^{\alpha}(t)$.
		
	\end{proof}
	\subsection{Asymptotic compactness}
	
	First, we decompose the solution $$U(t)=S^{\alpha}(t)U_{0}=(u(t), u_{t}(t),\eta
	)~\mbox{with}~ U_{0}=(u_{0}, u_{1},\eta_{0})\in \mathcal{B}_{\alpha}$$ of \eqref{1.2} into the sum
	\begin{align*}
		U(t)=S^{\alpha}(t)U_{0}=S_{1}^{\alpha}(t)U_{0}+S_{2}^{\alpha}(t)U_{0}=\widehat{U}_{1}(t)+\widehat{U}_{2}(t),
	\end{align*}
	where $\widehat{U}_{1}(t)=(\widehat{v},\widehat{v}_{t},\widehat{\xi}), \widehat{U}_{2}(t)=(\widehat{w},\widehat{w}_{t},\widehat{\zeta})$ are the solutions, respectively, to the problems
	\begin{equation}\label{3.3}
		\left\{
		\begin{array}{l}
			\widehat{v}_{tt}+kA^{2}\widehat{v}+A^{\alpha}\widehat{v}_{t}+\displaystyle\int_{0}^{\infty}g(s)A^{2}\widehat{\xi}(x,t,s)ds+f(\widehat{v})+L \widehat{v}=0, (x,t)\in \Omega\times R^{+},\\
			\widehat{\xi}_{t}=-\widehat{\xi}_{s}+\widehat{v}_{t}, (x,t,s)\in \Omega\times R^{+}\times R^{+},\\
			\widehat{v}=\widehat{\xi}=\Delta \widehat{v}=\Delta\widehat{\xi}=0,  (x,t,s)\in \Gamma \times R^{+}\times R^{+},\\
			( \widehat{v}(x,0), \widehat{v}_{t}(x,0),\widehat{\xi}(x,0,s))=(u_{0},u_{1}, \eta_{0}), x\in\Omega, s\in R^{+},
		\end{array}
		\right.
	\end{equation}
	and
	\begin{equation}\label{3.4}
		\left\{
		\begin{array}{l}
			\widehat{w}_{tt}+kA^{2}\widehat{w}+A^{\alpha}\widehat{w}_{t}+\displaystyle\int_{0}^{\infty}g(s)A^{2}\widehat{\zeta}(x,t,s)ds+f(u)-f(\widehat{v})\\
~~=h(x)+L\widehat{v}, (x,t)\in \Omega\times R^{+},\\
			\widehat{\zeta}_{t}=-\widehat{\zeta}_{s}+\widehat{w}_{t}, (x,t,s)\in \Omega\times R^{+}\times R^{+},\\
			\widehat{w}=\widehat{\zeta}=\Delta\widehat{w}=\Delta\widehat{\zeta}=0,  (x,t,s)\in \Gamma \times R^{+}\times R^{+},\\
			(\widehat{w}(x,0),\widehat{w}_{t}(x,0), \widehat{\zeta}(x,0,s))=0, x\in\Omega, s\in R^{+},
		\end{array}
		\right.
	\end{equation}
	where $L>0$ is a constant to be determined. By the Faedo-Galerkin method, one can prove the existence and uniqueness of solutions to the problems \eqref{3.3} and \eqref{3.4}. Next, we give some estimates for solutions of \eqref{3.3} and \eqref{3.4}.
	\begin{lemma}\label{lem4.2}
		Under The Assumption \ref{ass1} and the condition \eqref{assp}, for $U_{0}=(u_{0}, u_{1},\eta_{0})\in \mathcal{B}_{\alpha}$, the solutions of \eqref{3.3}  and \eqref{3.4} satisfy
		\begin{align}\label{3.5}
			\sup_{U_{0}\in \mathcal{B}_{\alpha}, t\in R^{+}}\|S_{1}^{\alpha}(t)U_{0}\|_{\mathcal{H}}^{2}\le C(\|U_{0}\|_{\mathcal{H}})e^{-\gamma_{1}t},
		\end{align}
		for a constant $\gamma_{1}>0$, and
		\begin{align}\label{3.6}
			\sup_{U_{0}\in \mathcal{B}_{\alpha}, t\in R^{+}} \|S_{2}^{\alpha}(t)U_{0}\|_{\mathcal{H}^{\alpha}}^{2}\le C(\|U_{0}\|_{\mathcal{H}}).
		\end{align}
	\end{lemma}

	\begin{proof}
	Using \eqref{f2} and noticing $V_{2}\hookrightarrow L^{\frac{(p-1)n}{4-\varepsilon_{0}}}$ for $n>4 $ ($1\le p<p^{*}$) and $V_{2}\hookrightarrow L^{\infty}$ for $n=1,2,3,4$, similarly as in the proof of Lemma \ref{lem4.1} we infer, for any $0<\varepsilon_{0}<<1$,
		\begin{align*}
			\int_{\Omega}f(\widehat{v})\widehat{v}dx\le& C\int_{\Omega}\left(1+|\widehat{v}|^{p-1}\right)|\widehat{v}|^{2}dx\nonumber\\
			\le&C\|\widehat{v}\|_{L^{\frac{2n}{n-4}}}\|\widehat{v}\|_{L^{\frac{2n}{n-2(2-\varepsilon_{0})}}}\left(1+\|\widehat{v}\|_{L^{\frac{(p-1)n}{4-\varepsilon_{0}}}}\right)\le\frac{k}{4}\|\widehat{v}\|_{V_{2}}^{2}+C\|\widehat{v}\|^{2},
		\end{align*}
		also
		\begin{align*}
			\int_{\Omega}F(\widehat{v})dx\le& C\int_{\Omega}\left(1+|\widehat{v}|^{p-1}\right)|\widehat{v}|^{2}dx\le\frac{k}{4}\|\widehat{v}\|_{V_{2}}^{2}+C\|\widehat{v}\|^{2}.
		\end{align*}
		Take $L$ large enough and $\varepsilon$ small enough, and note $\mathcal{F}=f(\widehat{v})+L\widehat{v}$. Then the estimate \eqref{3.5} follows.
		
		Multiply the equation in \eqref{3.4} by $A^{\alpha}\widehat{w}_{t}+\varepsilon A^{\alpha+\sigma_{0}}\widehat{w}$ for $\sigma_{0}=\min\{\alpha, 2-\alpha\}$, and integrate over $\Omega$. Replacing $\varphi$ with $\widehat{w}$ in \eqref{P5}--\eqref{P8}, we see
		\begin{align}\label{3.7}
			&\frac{d}{dt}\left(E_{\alpha}(\widehat{w})+\varepsilon I_{\alpha+\sigma_{0}}(\widehat{w})+\frac{\varepsilon}{2}J_{\alpha+\sigma_{0}}(\widehat{w})\right)+\|\widehat{w}_{t}\|_{V_{2\alpha}}^{2}+k\varepsilon\|\widehat{w}\|_{V_{2+\alpha+\sigma_{0}}}^{2}+\frac{\varepsilon}{2}\|\widehat{\zeta}\|_{\mathcal{M}_{2+\alpha+\sigma_{0}}}^{2}\nonumber\\
			=&\varepsilon\|\widehat{w}_{t}\|_{V_{\sigma_{0}+\alpha}}^{2}-\int_{\Omega}\left(f(u)-f(\widehat{v})\right)(A^{\alpha}\widehat{ w}_{t}+\varepsilon A^{\alpha+\sigma_{0}}\widehat{w})dx\nonumber\\
			&+\int_{\Omega}(h(x)+L\widehat{v})(A^{\alpha} \widehat{ w}_{t}+\varepsilon A^{\alpha+\sigma_{0}}\widehat{w})dx.
		\end{align}
		By the H\"older inequality, we have
		\begin{align}\label{3.8}
			&\int_{\Omega}(h(x)+L\widehat{v})(A^{\alpha} \widehat{w}_{t}+\varepsilon A^{\alpha+\sigma_{0}}\widehat{w})dx\nonumber\\
			\le& \frac{1}{4}\|\widehat{w}_{t}\|_{V_{2\alpha}}^{2}+ \frac{k\varepsilon}{4}\|\widehat{w}\|_{V_{2+\alpha+\sigma_{0}}}^{2}+C(\varepsilon)\left(\|h\|^{2}+\|\widehat{v}\|^{2}\right).
		\end{align}
		Also, using \eqref{f2} and \eqref{assp} gives that, for any $0<\varepsilon_{0}<<\alpha$,
		\begin{align}\label{3.9}
			&-\int_{\Omega}(f(u)-f(\widehat{v}))(A^{\alpha} \widehat{w}_{t}+\varepsilon A^{\alpha+\sigma_{0}}\widehat{w})dx\nonumber\\
			\le&C\int_{\Omega}\left(1+|u|^{p-1}+|\widehat{v}|^{p-1}\right)|\widehat{w}|\left(|A^{\alpha}\widehat{w}_{t}|+\varepsilon |A^{\alpha+\sigma_{0}}\widehat{w}|\right)dx\nonumber\\
			\le& C\|\widehat{w}\|_{L^{\frac{2n}{n-2(2+\alpha+\sigma_{0}-\varepsilon_{0})}}}\left(\|A^{\alpha}\widehat{w}_{t}\|+\varepsilon \|A^{\alpha+\sigma_{0}}\widehat{w}\|\right)\left(1+\|u\|_{L^{\frac{(p-1)n}{2+\alpha+\sigma_{0}-\varepsilon_{0}}}}^{p-1}+\|\widehat{v}\|_{L^{\frac{(p-1)n}{2+\alpha+\sigma_{0}-\varepsilon_{0}}}}^{p-1}\right)\nonumber\\
			\le& \frac{1}{4}\|\widehat{w}_{t}\|_{V_{2\alpha}}^{2}+\frac{k\varepsilon}{4}\|\widehat{w}\|_{V_{2+\alpha+\sigma_{0}}}^{2}+C(\varepsilon)
\|\widehat{w}\|^{2},
		\end{align}
by virtue of the interpolation theorem. Inserting \eqref{3.8} and \eqref{3.9} into \eqref{3.7}, using \eqref{P9}, taking $\varepsilon$ small enough, and then employing the Gronwall inequality, we verify \eqref{3.6}.
	\end{proof}
	The estimate \eqref{3.6} guarantees that $S^{\alpha}(t)\mathcal{B}_{\alpha}$ is bounded in $\mathcal{H}^{\alpha}$, however, the embedding $\mathcal{H}^{\alpha} \subset \mathcal{H}$ is not compact (cf. \cite{Pata-2001}). To address this, we turn to the following criterion for compactness of memory terms.
	\begin{lemma}[\cite{Pata-2001}, Lemma 5.5]\label{lemq}
		Let $g\in C(R^{+})\cap L^{1}(R^{+})$ be a non-negative function, such that $g(s)=0$ whenever $g(s_{0})=0$ and $s> s_{0}$, for some $s_{0}\in R^{+}$.
		Let $X_{0}, X, X_{1}$ be three Hilbert spaces such that
		\begin{align*}
			X_{0}\hookrightarrow X\hookrightarrow X_{1},
		\end{align*}
		the first injection being compact. Let $\mathcal{C} \subset L_{g}^{2}(R^{+}, X) $ satisfy the following hypotheses:
		
		\begin{enumerate}
			\item[\rm (i)] $\mathcal{C}$ is bounded in $L_{g}^{2}(R^{+}, X_{0})\cap H_{g}^{1}(R^{+}, X_{1})$;
			\item[\rm(ii)]$\sup_{\eta\in \mathcal{C}}\|\eta(s)\|_{X}^{2}\le K(s)$, for some $K\in L_{g}^{1}(R^{+})$.
		\end{enumerate}
		Then $\mathcal{C}$ is relatively compact in $ L_{g}^{2}(R^{+}, X)$.
	\end{lemma}
	We are now ready to prove the existence of attractor.
	\begin{theorem}\label{them3.3}
		Let The Assumption \ref{ass1} hold and assume \eqref{assp1}, or \eqref{assp} if $\Gamma \in C^{\infty}.$		
		Then,  for $\alpha\in(0,2\rbrack$, the semigroup $S^{\alpha}(t)$ on $\mathcal{H}$ possesses a global attractor $\mathcal{A}_{\alpha}$, and $\mathcal{A}_{\alpha}=\omega(\mathcal{B}_{\alpha})$ is contained and  uniformly bounded in $\mathcal{H}^{\alpha}$.
	\end{theorem}
	\begin{proof}
		We set
		\begin{align*}
			\mathcal{C} =\bigcup_{U_{0}\in \mathcal{B}_{\alpha}}\widehat{\zeta} \subset \mathcal{M}_2.
		\end{align*}
		From \eqref{3.4}, we have
		\begin{equation*}
			\widehat{\zeta}(\cdot,t,s)=\left\{
			\begin{array}{lll}
				\widehat{w}(\cdot,t)-\widehat{w}(\cdot,t-s),&t\ge s\ge 0,\\[0.28cm]
					\widehat{w}(\cdot,t),&s>t.
			\end{array}
			\right.
		\end{equation*}
		This shows
		\begin{align*}
			\|\widehat{\zeta}_{s}\|_{\mathcal{M}_{\alpha}}^{2}=&\int_{0}^{t}g(s)\|\partial_{s}\widehat{w}(t-s)\|_{\alpha}^{2}ds
			\le\int_{0}^{t}g(s)ds\sup_{0\le s\le t}\|\partial_{s}\widehat{w}(s)\|_{\alpha}^{2}<\infty.
		\end{align*}
		 It follows from \eqref{3.6} that
		$\mathcal{C} \subset L^2_g(R^+; V_{2+\alpha}) \cap H_g^1(R^+; V_{\alpha})$. On the other hand, we have
		\begin{align*}
			\sup_{\widehat{\zeta}\in \mathcal{C}} \|\Delta\widehat{\zeta}\|^{2}\le \sup_{0\le s\le t}\|\Delta \widehat{w}(t)-\Delta \widehat{w}(t-s)\|^{2}+ \sup_{s\ge t}\|\Delta \widehat{w}(t)\|^{2}\le 3\sup_{t\in R^{+}}\|\Delta \widehat{w}\|^{2}<\infty.
		\end{align*}
		Utilizing Lemma \ref{lemq} with the embedding relation $V_{2+\alpha} \hookrightarrow V_2 \hookrightarrow V_\alpha$, we deduce that $\mathcal{C}$ is relatively compact in $\mathcal{M}_2$. Let $\overline{\mathcal{C}}$ denote the closure of $\mathcal{C}$ in $\mathcal{M}_2$. Also, denote by $B_{\mathcal{R}_0}$ the closed ball of $V_{2+\alpha} \times V_\alpha$ centered at the origin with radius $\mathcal{R}_0$, satisfying
		\begin{align*}
			\sup_{U_{0}\in \mathcal{B}_{\alpha}}\|(\widehat{w},\widehat{w}_{t})\|_{V_{2+\alpha}\times V_{\alpha}}\le \mathcal{R}_{0},
		\end{align*}
		and write
		\begin{align}\label{K}
			\mathcal{K}_{\alpha}=B_{\mathcal{R}_{0}}\times \overline{\mathcal{C}}\subset \mathcal{H}^{\alpha}.
		\end{align}
 Clearly, $B_{\mathcal{R}_0}$ is compact subset in $V_2 \times V_0$, because the embedding $V_{2+\alpha} \times V_\alpha \hookrightarrow V_2 \times V_0$ is compact. This means that $\mathcal{K}_{\alpha}$ is compact in $\mathcal{H}$, and hence
the operators $ S_2^\alpha(t)$ are uniformly compact for $t$ large. Thus applying \cite[Theorem I.1.1]{Temam-1988}, we conclude that $\mathcal{A}_{\alpha}:=\omega(\mathcal{B}_{\alpha})$ is a  global attractor.
		
		Furthermore, for any bounded subset $B \subset \mathcal{H}$, there exists $t_B > 0$ such that for all $t \ge t_B$, we have $S^\alpha(t) B \subset \mathcal{B}_\alpha$. For $t > t_B$ and $t_0 = t - t_B$, using the semigroup property we obtain
		\begin{align}\label{4.350}
			S^{\alpha}(t)B=S^{\alpha}(t_{B}+t_{0})B\subset S^{\alpha}(t_{0})S^{\alpha}(t_{B})B\subset S^{\alpha}(t_{0})\mathcal{B}_{\alpha}.
		\end{align}
		Consequently, for any $S^\alpha(t) U_0 = (u, u_t, \eta) \in S^\alpha(t) B$ with $t > t_B$, it follows from \eqref{3.6} and \eqref{4.350} that the solution $S_2^\alpha(t) U_0$ to the problem \eqref{3.4} belongs to $ \mathcal{K}_{\alpha}$. Therefore,
		\begin{align*}
			\inf_{z \in \mathcal{K}_{\alpha}} \| S^\alpha(t) U_0 - z \|_{\mathcal{H}} \le \inf_{z \in \mathcal{K}_{\alpha}} \| S_2^\alpha(t) U_0 - z \|_{\mathcal{H}} + \| S_1^\alpha(t) U_0 \|_{\mathcal{H}} = \| S_1^\alpha(t) U_0 \|_{\mathcal{H}},
		\end{align*}
		for $t > t_B$. Using \eqref{3.5}, we obtain
		\begin{align*}
			\text{dist}_{\mathcal{H}}(S^\alpha(t) B, \mathcal{K}_{\alpha}) \le \sup_{U_0 \in \mathcal{B}_\alpha} \| S_1^\alpha(t) U_0 \|_{\mathcal{H}} \le C(\| U_0 \|_{\mathcal{H}}) e^{-\frac{\gamma_1}{2} t}, \quad t > t_B.
		\end{align*}
		Since $\mathcal{A}_{\alpha}$ is fully invariant, it is contained in every closed attracting set, that is,  $\mathcal{A}_{\alpha} \subset \mathcal{K}_{\alpha}$. Therefore, $\mathcal{K}_{\alpha}$ and hence $\mathcal{A}_{\alpha}$ are bounded in $\mathcal{H}^{\alpha}$, and the proof is complete.
	\end{proof}
	\begin{remark}{\rm For the case of weak damping (\(\alpha = 0\)), we cannot obtain the compact embedding result here. The proof of the  existence of attractor will be given in the next section, by the use of quasi-stability.}

	\end{remark}
	\section{Regularity of the attractor}
		
	To facilitate the presentation of the subsequent results, we first introduce the requisite definitions and foundational theoretical concepts.
	The following is the definition of quasi-stability.
	\begin{definition} [quasi-stability]\label{def4.2}
		Let $X, Y, Z$ be reflexive Banach spaces, with $X$ compactly embedded in $Y$. Let H be the product space $X \times Y \times Z$. Assume that $(H, S(t))$ is a dynamical system on $H$ taking the form
		\begin{equation*}
			S(t)U=(u(t),u_{t}(t), \eta),~ \forall t\ge0, U\in H,
		\end{equation*}
		where
		\begin{equation*}
			u\in C(R^{+},X)\cap C^{1}(R^{+},Y), ~\eta\in C(R^{+},Z).
		\end{equation*}
		Dynamical system $(H,S(t))$ is said to be (asymptotically) quasi-stable on a set $B\subset H$ if there exist a compact semi-norm $n_{X}(\cdot)$ on the space $X$, non-negative scalar functions $a(t), c(t) \in L_{loc}^{\infty}(R^{+})$, and $b(t)\in L^{1}(R^{+})$ with $\lim_{t\to\infty} b(t)=0$ such that, for every $y_{1}, y_{2}\in B$ and $t>0$,
		\begin{equation*}
			\|S(t)y_{1}-S(t)y_{2}\|_{H}^{2}\le a(t)\|y_{1}-y_{2}\|_{H}^{2},
		\end{equation*}
		and
		\begin{equation}\label{quasi-stable}
			\|S(t)y_{1}-S(t)y_{2}\|_{H}^{2}\le b(t)\|y_{1}-y_{2}\|_{H}^{2}+c(t)\sup_{0\le s\le t}\left[n_{X}\left(u^{1}(s)-u^{2}(s)\right)\right]^{2}.
		\end{equation}
		Here $S(t)y_{i}:=(u^{i}(t), u_{t}^{i}(t), \eta^{i}(t)),~ i=1,2$.
	\end{definition}
	The next lemma, concerning the regularity of time derivatives for any full trajectory in the attractor, can be proved
similarly as in the proof of \cite[Theorem 4.17]{Chueshov-2013}.	

\begin{lemma}\label{lemma4.3}
		Let the assumptions in Definition \ref{def4.2} be satisfied. Assume that the dynamical system $(H, S(t))$ has a global attractor $\mathcal{A}$ and is quasi-stable on $\mathcal{A}$.
		 Let \(b(t)\) and \(c(t)\) be the functions as defined in Definition \ref{def4.2}, and suppose that $c_{\infty} = \sup_{t \in \mathbb{R}^+} c(t) < \infty$. For a constant $\theta_0 \geq 0$, if
			\begin{align}\label{qs}
				\|S(t)y_1 - S(t)y_2\|_H^2 \leq Cb(t) + c(t) \sup_{0 \leq s \leq t} \left[n_X(u^1(s) - u^2(s))\right]^{\theta_0},
			\end{align}
			then any full trajectory $\{(u, u_t, \eta) : t \in \mathbb{R}\}$ belonging to the global attractor satisfies
			\begin{align*}
				u_{t}\in L^{\infty}(R;X)\cap C(R;Y),~u_{tt}\in L^{\infty}(R;Y),~\eta_{t}\in L^{\infty}(R;Z);
			\end{align*}
			furthermore, there exists a constant \(r > 0\) such that
			\begin{align*}
				\|u_t\|_X^2 + \|u_{tt}\|_Y^2 + \|\eta_t\|_Z^2 \leq r^2, \quad \text{for all } t \in R,
			\end{align*}
			where \(r\) depends on the constant $c_{\infty}$.
		\end{lemma}
	\begin{remark}\label{re4.4}{\rm
As shown in \cite[Theorem 4.17]{Chueshov-2013}, the quasi-stability inequality \eqref{quasi-stable} (under the condition \(\sup_{t \in \mathbb{R}^+} c(t) < \infty\)) ensures the regularity of the time derivatives, that was applied in \cite{Miranda-2025} among others. In the following, we will exploit the more general inequality \eqref{qs} (in place of \eqref{quasi-stable}) to deal with the regularity problem for the trajectories.}
	\end{remark}
	\subsection{Quasi-stability}
	
	We commence by establishing a lemma to demonstrate that the system $(\mathcal{H}, S^{\alpha}(t))$ generated by problem \eqref{1.2} exhibits quasi-stability on the attractor $\mathcal{A}_{\alpha}$.
	\begin{lemma}\label{lem5}
		Let The Assumption \ref{ass1} and \eqref{assp1} hold for $\alpha\in[0,2]$. Then, the semigroup $S^{0}(t)$ on $\mathcal{H}$ has a global attractor $\mathcal{A}_{0}$ (note $p^{**}=p^{*}$ for $\alpha=0$).
 Furthermore, for each $\alpha\in[0,2]$, the
		dynamical system $(\mathcal{H}, S^{\alpha}(t))$ is quasi-stable on $\mathcal{A}_{\alpha}$. In addition, for any  full trajectory $\left\{(u, u_t, \eta) | t \in \mathbb{R}\right\}$ of the global attractors we have the following regular estimate regarding the time derivative
		\begin{align}\label{5.6}
			\|u_{t}\|_{V_{2+\alpha}}^{2}+\|u_{tt}\|_{V_{\alpha}}^{2}+\|\eta_{t}\|_{\mathcal{M}_{2+\alpha}}^{2}\le C.
		\end{align}
	\end{lemma}
	\begin{proof}
		Replacing $\varphi$ with $\widehat{u}$ in \eqref{5.1} and taking $\mathcal{F}=f(u^{1})-f(u^{2})$, from \eqref{P5}-- \eqref{P8} we have
		\begin{align}\label{4.9}
			&\frac{d}{dt}\left(E_{\sigma_{1}}(\widehat{u})+\varepsilon I_{\sigma_{2}}(\widehat{u})+\frac{\varepsilon}{2}J_{\sigma_{2}}(\widehat{u})\right)+\|\widehat{u}_{t}\|_{V_{\alpha+\sigma_{1}}}^{2}-\varepsilon\|\widehat{u}_{t}\|_{V_{\sigma_{2}}}^{2}+k\varepsilon\|\widehat{u}\|_{V_{2+\sigma_{2}}}^{2}+\frac{\varepsilon}{2}\|\widehat{\eta}\|_{\mathcal{M}_{2+\sigma_{2}}}^{2}\nonumber\\
			\le &-\int_{\Omega}(f(u^{1})-f(u^{2}))(A^{\sigma_{1}}\widehat{u}_{t}+\varepsilon A^{\sigma_{2}}\widehat{u})dx.
		\end{align}
		Take $\sigma_{1}=\sigma_{2}=\sigma$ ($0\le \sigma\le\alpha$) in \eqref{4.9}. It follows from the interpolating inequality and \eqref{assp1} that for  $0<\varepsilon_{0}<<1$,
		\begin{align}\label{ff}
			&-\int_{\Omega}(f(u^{1})-f(u^{2}))(A^{\sigma}\widehat{u}_{t}+\varepsilon A^{\sigma}\widehat{u})dx\nonumber\\
			\le& C \int_{\Omega}(1+|u^{1}|^{p-1}+|u^{2}|^{p-1})|\widehat{u}|(|A^{\sigma}\widehat{u}_{t}|+\varepsilon |A^{\sigma}\widehat{u}| )dx\nonumber\\
			\le& C\|\widehat{u}\|_{L^{\frac{2n}{n-2(2+\sigma-\varepsilon_{0})}}}\left(\|A^{\sigma}\widehat{u}_{t}\|_{L^{\frac{2n}{n-2(\alpha-\sigma)}}}+\varepsilon\|A^{\sigma}\widehat{u}\|_{L^{\frac{2n}{n-2(2-\sigma)}}}\right)\nonumber\\
			&\times \left(1+\|u^{1}\|_{L^{\frac{(p-1)n}{2+\alpha-\varepsilon_{0}}}}^{p-1}+\|u^{2}\|_{L^{\frac{(p-1)n}{2+\alpha-\varepsilon_{0}}}}^{p-1}+\|u^{1}\|_{L^{\frac{(p-1)n}{4-\varepsilon_{0}}}}^{p-1}+\|u^{2}\|_{L^{\frac{(p-1)n}{4-\varepsilon_{0}}}}^{p-1}\right)\nonumber\\
			\le& \frac{1}{4}\|\widehat{u}_{t}\|_{V_{\alpha+\sigma}}^{2}+\frac{k\varepsilon}{2}\|\widehat{u}\|_{V_{2+\sigma}}^{2}+C(\varepsilon)\|\widehat{u}\|^{2}.
		\end{align}
		 Choosing $\varepsilon$ small enough, we have
		\begin{align}\label{5.2}
			\frac{d}{dt}\left(E_{\sigma}(\widehat{u})+\varepsilon I_{\sigma}(\widehat{u})+\frac{\varepsilon}{2}J_{\sigma}(\widehat{u})\right)+\frac{1}{2}\|\widehat{u}_{t}\|_{V_{\alpha+\sigma}}^{2}+\frac{k\varepsilon}{2}\|\widehat{u}\|_{V_{2+\sigma}}^{2}+\frac{\varepsilon}{2}\|\widehat{\eta}\|_{\mathcal{M}_{2+\sigma}}^{2}
			\le C\|\widehat{u}\|^{2}.	
		\end{align}
		
For the case $\alpha\in (0,2]$, from \eqref{P9} and \eqref{5.2} and by the use of the Gronwall inequality, we deduce that
 for $z^{i}=(u_{0}^{i}, u_{1}^{i}, \eta_{0}^{i})\in \mathcal{A}_{\alpha}$, $i=1,2$,
		\begin{equation*}
			\|S^{\alpha}(t)z^{1}-S^{\alpha}(t)z^{2}\|_{\mathcal{H}^{\sigma}}^{2}\le Ce^{\gamma_{2}t}\|z^{1}-z^{2}\|_{\mathcal{H}^{\sigma}}^{2},
		\end{equation*}
		and for any $0\le \sigma\le \alpha$,
		\begin{equation*}
			\|S^{\alpha}(t)z^{1}-S^{\alpha}(t)z^{2}\|_{\mathcal{H}^{\sigma}}^{2}\le Ce^{-\gamma_{3}t}\|z^{1}-z^{2}\|_{\mathcal{H}^{\sigma}}^{2}+C\left(1-e^{-\gamma_{3}t}\right)\sup_{s\in[0,t]}\|u^{1}-u^{2}\|^{2},
		\end{equation*}
		with some $\gamma_{2}, \gamma_{3}>0$.
 This implies that the dynamical system $(\mathcal{H}, S^{\alpha}(t))$ is quasi-stable on $\mathcal{A}_{\alpha}$.
For the case $\alpha=0$, consider a bounded set $B\subset\mathcal{H}$ and $z^{i}=(u_{0}^{i}, u_{1}^{i}, \eta_{0}^{i})\in B$. Similarly, we see that the dynamical system $(\mathcal{H}, S^{\alpha}(t))$ is quasi-stable on $B$. Therefore, an application of Lemma \ref{lem4.1} leads to the existence of the attractor \(\mathcal{A}_0\) and so the quasi-stability of the dynamical system on $\mathcal{A}_0$.

Moreover, the estimate \eqref{5.6} follows from Lemma \ref{lemma4.3}, and this ends the proof.
	\end{proof}

	\subsection{Regularity}
	We now present our main result.
	\begin{theorem}\label{them4.4}
		Let the hypothesis in Theorem \ref{them3.3} hold for $\alpha \in [0, 2]$.
		Then, the global attractor $\mathcal{A}_{\alpha}$ of the dynamical system $(\mathcal{H}, S^{\alpha}(t))$ is contained and bounded in $\mathcal{Z} := V_4 \times V_{4-\delta_0} \times \mathcal{M}_4$, for any $0 < \delta_0 \ll 2 - \alpha$ with $\alpha\in\lbrack 0, 2)$ and $\delta_{0}=0$ with $\alpha=2$ .
	\end{theorem}
	\begin{remark}\label{re}{\rm
	
	The proof method we will display below for Theorem \ref{them4.4} may be applied to wave equations with fractional damping and memory, giving the regularity space of attractors
 $V_2 \times V_{2-\epsilon}\times \mathcal{M}_2$  ($\epsilon$ can be an arbitrarily small positive constant, and 0 in the case of strong damping), if the phase space is $V_1 \times V_{0}\times \mathcal{M}_1$; as a byproduct this would improve on the existing regularity space $V_2 \times V_1 \times \mathcal{M}_2 $ for the weak damping $u_t$ (\cite{Ma-2004}) or the strong damping $-\Delta u_t$ (\cite[p.778]{Di Plinio-2008} and \cite{Yang-2024}). Also, this method may improve the regularity results for attractors of wave equations with fractional damping (cf. \cite{Li-2020,Li-2021}).}
	\end{remark}

	\subsubsection{Proof of Theorem \ref{them4.4} for $\Gamma\in C^{2}$}

 Building upon regularity results for the trajectories with respect to their time derivatives in Lemma \ref{lem5}, we further decompose the problem in \eqref{1.2} and provide higher-order estimates, ultimately deriving higher regularity results for the attractor. 	

 For initial data on the attractor, i.e. $U_{0}=(u_{0},u_{1}, \eta_{0})\in \mathcal{A}_{\alpha}$, we employ once more the decomposition of the solution $S^{\alpha}(t)U_{0}=S_{3}^{\alpha}(t)U_{0}+S_{4}^{\alpha}(t)U_{0}$, where $S_{3}^{\alpha}(t)U_{0}=(v, v_{t},\xi),~ S_{4}^{\alpha}(t)U_{0}=(w,w_{t},\zeta)$, satisfying (instead of \eqref{3.3} and \eqref{3.4}), respectively,
	\begin{equation}\label{3.10}
		\left\{
		\begin{array}{l}
			v_{tt}+kA^{2}v+A^{\alpha}v_{t}+\displaystyle\int_{0}^{\infty}g(s)A^{2}\xi(x,t,s)ds=0, ~(x,t)\in \Omega\times R^{+},\\
			\xi_{t}=-\xi_{s}+v_{t},~(x,t,s)\in \Omega\times R^{+}\times R^{+},\\
			v=\xi=\Delta v=\Delta \xi=0, ~(x,t,s)\in \Gamma\times R^{+}\times R^{+},\\
			S_{3}(0)U_{0}=U_{0}, ~x\in \Omega, s\in R^{+},
		\end{array}
		\right.
	\end{equation}
	and
	\begin{equation}\label{3.11}
		\left\{
		\begin{array}{l}
			w_{tt}+kA^{2}w+A^{\alpha}w_{t}+\displaystyle\int_{0}^{\infty}g(s)A^{2}\zeta(x,t,s)ds+f(u)=h(x), ~(x,t)\in \Omega\times R^{+},\\
			\zeta_{t}=-\zeta_{s}+w_{t}, ~(x,t,s)\in \Omega\times R^{+}\times R^{+},\\
			w=\zeta=\Delta w=\Delta \zeta=0,~(x,t,s)\in \Gamma\times R^{+}\times R^{+}, \\
			S_{4}(0)U_{0}=0, ~x\in\Omega, s\in R^{+}.
		\end{array}
		\right.
	\end{equation}
	For the above decomposition, we provide the following crucial proposition, which, by invoking Lemma \ref{lemq} and Remark A.5 in \cite{Conti-2005-}, directly leads to the regularity of the attractor in Theorem \ref{them4.4}.
	\begin{proposition}\label{lem3.4}
		Let The Assumption \ref{ass1} and \eqref{assp1} hold. Then, for the solution of the problem \eqref{3.10} and \eqref{3.11}, we get the following estimates
		\begin{equation}\label{3.14}
			\|S_{3}^{\alpha}(t)U_{0}\|_{\mathcal{H}^{\alpha}}^{2}\le C(\|U_{0}\|_{\mathcal{H}^{\alpha}}
			)e^{-\gamma_{4}t}
		\end{equation}
		for some $\gamma_{4}>0$, and
		\begin{equation}\label{3.25}
			\|(w,w_{t},\zeta)\|_{\mathcal{H}^{2}}^{2}+\|w_{tt}\|_{V_{2-\delta_{0}}}^{2}+\|w_{t}\|_{V_{4-\delta_{0}}}^{2}+\|\zeta_{t}
\|_{\mathcal{M}_{4-\delta_{0}}}^{2}\le C(\|U_{0}\|_{\mathcal{H}^{2}}).
		\end{equation}
	\end{proposition}
	\begin{proof}
		Replacing $\varphi$ with $v$, and taking $\sigma_{1}=\sigma_{2}=\alpha$ and $\mathcal{F}=0$ in \eqref{P5}--\eqref{P8} yields \eqref{3.14}, as a particular case of Lemma \ref{lem4.2}.
		
		Writing $q=w_{t}$, we have
		\begin{align}\label{3.15+}
			q_{tt}+A^{\alpha}q_{t}+kA^{2}q+\displaystyle\int_{0}^{\infty}g(s)A^{2}\zeta_{t}ds+f^{\prime}(u)u_{t}=0.
		\end{align}
		Replacing $\varphi$ by $q$ in \eqref{P5}--\eqref{P8}, and taking $\mathcal{F}=f^{\prime}(u)u_{t}$,  we see
		\begin{align}\label{3.16}
			&	\frac{d}{dt}\left(E_{\sigma_{1}}(q)+\varepsilon I_{\sigma_{2}}(q)+\frac{\varepsilon}{2}J_{\sigma_{2}}(q)\right)+\|q_{t}\|_{V_{\alpha+\sigma_{1}}}^{2}+\frac{\varepsilon}{2}\|\zeta_{t}\|_{\mathcal{M}_{2+\sigma_{2}}}^{2}+k\varepsilon\|q\|_{V_{2+\sigma_{2}}}^{2}\nonumber\\
			\le&\varepsilon\|q_{t}\|_{V_{\sigma_{2}}}^{2}-\int_{\Omega}f^{\prime}(u)u_{t}\left(A^{\sigma_{1}}q_{t}+\varepsilon A^{\sigma_{2}}q\right)dx.
		\end{align}
		Take $\sigma_{1}= \alpha$ and $\sigma_{2}=\min\{2\alpha,2\}$. It follows from Theorem \ref{them3.3}, \eqref{5.6},  $u, u_{t}\in V_{2+\alpha}$, and \eqref{assp1} that
		\begin{align}\label{3.17}
			&	-\int_{\Omega}f^{\prime}(u)u_{t}(A^{\alpha}q_{t}+\varepsilon A^{\sigma_{2}}q)dx\nonumber\\
			\le & C\int_{\Omega}(1+|u|^{p-1})|u_{t}|(|A^{\alpha}q_{t}|+\varepsilon |A^{\sigma_{2}}q|)dx\nonumber\\
			\le&C \|u_{t}\|_{L^{\frac{2n}{n-2(2+\alpha)}}}\left(\|A^{\alpha}q_{t}\|+\varepsilon\|A^{\sigma_{2}}q\|_{L^{\frac{2n}{n-2(2-\sigma_{2})}}}\right)\left(1+\|u\|_{L^{\frac{(p-1)n}{\alpha+2}}}^{p-1}\right)\nonumber\\
			\le& \frac{k\varepsilon}{2}\|q\|_{V_{2+\sigma_{2}}}^{2}+\frac{1}{2}\|q_{t}\|_{V_{2\alpha}}^{2}+C.
		\end{align}
		Inserting \eqref{3.17} into \eqref{3.16} and using the Gronwall inequality,  we obtain the estimate of the solution of \eqref{3.11} with respect to the time derivative:
		\begin{align}\label{3.18}
			\|w_{tt}\|_{V_{\alpha}}^{2}+\|w_{t}\|_{V_{2+\alpha}}^{2}+\|\zeta_{t}\|_{\mathcal{M}_{2+\alpha}}^{2}\le C.
		\end{align}
		Next, replacing $\varphi$ with $w$ in \eqref{P5}--\eqref{P8} gives
		\begin{align}\label{3.23}
			&\frac{d}{dt}\left(E_{\sigma_{1}}(w)+\varepsilon I_{\sigma_{2}}(w)+\frac{\varepsilon}{2}J_{\sigma_{2}}(w)\right)+\|w_{t}\|_{V_{\sigma_{1}+\alpha}}^{2}+\frac{\varepsilon}{2}\|\zeta\|_{\mathcal{M}_{2+\sigma_{2}}}^{2}+k\varepsilon\|w\|_{V_{2+\sigma_{2}}}^{2}\nonumber\\
			\le&\varepsilon\|w_{t}\|_{V_{\sigma_{2}}}^{2}-\int_{\Omega}(f(u)-h(x))\left(A^{\sigma_{1}}w_{t}+\varepsilon A^{\sigma_{2}}w\right)dx.
		\end{align}
		Taking  $\sigma_{1}=\frac{2+\alpha}{2}$  and $\sigma_{2}=2$, we have
		\begin{align}\label{3.19}
			\int_{\Omega}h(x)(A^{\sigma_{1}}w_{t}+\varepsilon A^{\sigma_{2}}w)dx\le \frac{1}{4}\|w_{t}\|_{V_{2+\alpha}}^{2}+\frac{k\varepsilon}{4}\|w\|_{V_{4}}^{2}+C,
		\end{align}
		and
		\begin{align}\label{3.20}
			&-\int_{\Omega}f(u)\left(A^{\sigma_{1}}w_{t}+\varepsilon A^{\sigma_{2}}w\right)dx\nonumber\\
			\le&C\|u\|_{L^{\frac{2n}{n-2(2+\alpha)}}}\left(\|w_{t}\|_{V_{2+\alpha}}+\varepsilon \|w\|_{V_{4}}\right)\left(1+\|u\|_{\frac{(p-1)n}{2+\alpha}}^{p-1}\right)\nonumber\\
			\le& \frac{k\varepsilon}{4}\|w\|_{V_{4}}^{2}+\frac{1}{4}\|w_{t}\|_{V_{2+\alpha}}^{2}+C.
		\end{align}
		Insert \eqref{3.19} and \eqref{3.20} into \eqref{3.23}, and take $\varepsilon$ small enough; it follows from \eqref{3.18} that
		\begin{align*}
			&\frac{d}{dt}\left(E_{\frac{2+\alpha}{2}}(w)+\varepsilon I_{2}(w)+\frac{\varepsilon}{2}J_{2}(w)\right)+\|w_{t}\|_{V_{\frac{2+\alpha}{2}+\alpha}}^{2}+\frac{\varepsilon}{2}\|\zeta\|_{\mathcal{M}_{4}}^{2}+\frac{k\varepsilon}{2}\|w\|_{V_{4}}^{2}\le C.
		\end{align*}
This combined with \eqref{P9} shows
		\begin{align}\label{3.24}
			\|(w,w_{t},\zeta)\|_{\mathcal{H}^{\frac{2+\alpha}{2}}}\le C.
		\end{align}
Estimates \eqref{3.18} and \eqref{3.24} imply that $S^{\alpha}(t)B$ is attracted by the set
		\begin{align*}
			\mathcal{K}:=\left\{(\widetilde{u},\widetilde{v},\widetilde{\eta}): \|\widetilde{u}\|_{V_{2+\frac{2+\alpha}{2}}}^{2}+\|\widetilde{v}\|_{V_{2+\alpha}}^{2}+\|\widetilde{\eta}_{t}\|_{\mathcal{M}_{2+\alpha}}^{2}+\|\widetilde{\eta}\|_{\mathcal{M}_{2+\frac{2+\alpha}{2}}}^{2}\le C\right\},
		\end{align*}
		which is compact in $\mathcal{H }$ by Lemma \ref{lemq}. Therefore,
we get $\mathcal{A}_{\alpha}\subset \mathcal{K}$.

	Next, let $\mathcal{V}_{1}=(v^{1}, v_{t}^{1}, \xi^{1})$ and $\mathcal{V}_{2}=(v^{2}, v_{t}^{2}, \xi^{2})$ be two solutions of problem \eqref{3.10} with initial values $z^{i}=(u_{0}^{i}, u_{1}^{i},\eta_{0}^{i})\in \mathcal{A}_{\alpha}, i=1,2$, respectively, and let $W_{1}=(w^{1}, w_{t}^{1}, \zeta^{1})$ and $W_{2}=(w^{2}, w_{t}^{2}, \zeta^{2})$ be the two solutions of the corresponding decomposition problem \eqref{3.11}. Then $W=(\phi, \phi_{t},\varsigma)=W_{1}-W_{2}$ satisfies the following equation
		\begin{align}\label{4.31}
			\phi_{tt}+kA^{2}\phi+A^{\alpha}\phi_{t}+\displaystyle\int_{0}^{\infty}g(s)A^{2}\varsigma ds+f(u^{1})-f(u^{2})=0,
		\end{align}
		where $u^{1}=w^{1}+v^{1},~ u^{2}=w^{2}+v^{2}$. Replacing $\varphi$ with $\phi$ in \eqref{P5}--\eqref{P8}, we get
		\begin{align}\label{4.9p}
			&\frac{d}{dt}\left(E_{\sigma_{1}}(\phi)+\varepsilon I_{\sigma_{2}}(\phi)+\frac{\varepsilon}{2}J_{\sigma_{2}}(\phi)\right)+\|\phi_{t}\|_{V_{\alpha+\sigma_{1}}}^{2}+k\varepsilon\|\phi\|_{V_{2+\sigma_{2}}}^{2}+\frac{\varepsilon}{2}\|\varsigma\|_{\mathcal{M}_{2+\sigma_{2}}}^{2}\nonumber\\
			\le &\varepsilon\|\phi_{t}\|_{V_{\sigma_{2}}}^{2}-\int_{\Omega}(f(u^{1})-f(u^{2}))(A^{\sigma_{1}}\phi_{t}+\varepsilon A^{\sigma_{2}}\phi)dx.
		\end{align}
	 It is easy to verify that, for $\gamma_{5}>0$,
		\begin{align}\label{4.30-}
			\|\mathcal{V}_{1}\|_{\mathcal{H}^{\frac{\alpha+2}{2}}}^{2}+ \|\mathcal{V}_{2}\|_{\mathcal{H}^{\frac{\alpha+2}{2}}}^{2}\le C\left(\|z^{1}\|_{\mathcal{H}^{\frac{2+\alpha}{2}}}+\|z^{2}\|_{\mathcal{H}^{\frac{2+\alpha}{2}}}\right)e^{-\gamma_{5} t}.
		\end{align}
	Now, taking $\sigma_{1}=\sigma_{2}=\theta:=\frac{2+\alpha}{2}$ in \eqref{4.9p}, by interpolation inequality we obtain, for any $0<\varepsilon_{0}<<1$,
		\begin{align}\label{4.27-}
			& -\int_{\Omega}\left(f(u^{1})-f(u^{2})\right)\left(A^{\frac{2+\alpha}{2}}\phi_{t}+\varepsilon A^{\frac{2+\alpha}{2}}\phi\right)dx\nonumber\\
			\le& C\left(\|\phi\|_{L^{\frac{2n}{n-2\left(2+\frac{2+\alpha}{2}-\varepsilon_{0}\right)}}}+\|v_{1}-v_{2}\|_{L^{\frac{2n}{n-2\left(2+\frac{2+\alpha}{2}-\varepsilon_{0}\right)}}}\right)\left(\left\|A^{\frac{2+\alpha}{2}}\phi_{t}\right\|+\varepsilon \left\|A^{\frac{2+\alpha}{2}}\phi\right\|\right)\nonumber\\
			&\times\left(1+\|u^{1}\|_{L^{\frac{(p-1)n}{2+\frac{2+\alpha}{2}-\varepsilon_{0}}}}^{p-1}+\|u^{2}\|_{L^{\frac{(p-1)n}{2+\frac{2+\alpha}{2}-\varepsilon_{0}}}}^{p-1}\right)\nonumber\\
			\le&\frac{k\varepsilon}{2}\|\phi\|_{V_{\frac{6+\alpha}{2}}}^{2}+C\left(\|v_{1}\|_{V_{\frac{6+\alpha}{2}}}^{2}+\|v_{2}\|_{V_{\frac{6+\alpha}{2}}}^{2}+\|v_{1}\|^{\mu_{1}}+\|v_{2}\|^{\mu_{1}}\right)+C\|\phi\|^{\mu_{1}},
		\end{align}
		where $\mu_{1}=\frac{4\varepsilon_{0}}{6+\alpha+2\varepsilon_{0}}$. Then inserting \eqref{4.27-} into \eqref{4.9p}, and using \eqref{4.30-} yields
		\begin{align*}
			\|W_{1}-W_{2}\|_{\mathcal{H}^{\frac{\alpha+2}{2}}}^{2}\le Ce^{-\gamma_{6} t}+C\left(1-e^{-\gamma_{6}t}\right)\sup_{s\in[0,t]}\|w^{1}-w^{2}\|^{\mu_{1}},
		\end{align*}
with some $\gamma_{6}>0$.
Thus, based on Lemma \ref{lemma4.3}, we obtain the regularity estimate for the time derivative $w_{t}\in V_{2+\theta}$ with $\theta=\frac{2+\alpha}{2}$. Following the procedure in \eqref{3.23}--\eqref{3.24} for $\sigma_{1}=\frac{2+\theta}{2}, \sigma_{2}=2$, we then derive the regularity space $V_{2+\frac{2+\theta}{2}}$ for $w$. By continuously repeating the above process and using the bootstrap argument, we find a monotonically increasing sequence $\{\theta_{m}\}_{m\ge0}$ with $\theta_{m}=1+\frac{1}{2}\theta_{m-1}, \theta_{0}=\alpha$ such that $w_{t}\in V_{2+\theta_{m}}$ and $w\in V_{2+\frac{2+\theta_{m}}{2}}$; that is,
		\begin{align*}
			\theta_{m}=\sum_{i=0}^{m-1}\left(\frac{1}{2}\right)^{i}+\left(\frac{1}{2}\right)^{m}\alpha.
		\end{align*}
It is evident that $\theta_{m}\to 2, \mbox{~as~}m\to\infty.$
Therefore, for any $0<\delta_{0}<<2-\alpha$, there exists a finite number $m_{0}$ such that $\theta_{m_{0}}=2-\delta_{0}$, for that, we have $\mathcal{A}_{\alpha}\subset V_{4-\delta_{0}}\times V_{4-\delta_{0}}\times \mathcal{M}_{4-\delta_{0}}$.
		
		Furthermore, by setting $\sigma_1 = \sigma_2 = 2$ in equation \eqref{3.23}, we can provide the following revised estimates for inequalities \eqref{3.19} and \eqref{3.20}:
		
		\begin{align}\label{3.19-}
			\int_{\Omega} h(x) \left( A^{2} w_t + \varepsilon A^{2} w \right) \, dx \leq \frac{d}{dt} \int_{\Omega} h(x) A^2 w \, dx + \frac{k\varepsilon}{4} \|w\|_{V_{4}}^2 + C,
		\end{align}
		and
		\begin{align}\label{3.21-}
			& - \int_{\Omega} f(u) \left( A^{2} w_t + \varepsilon A^{2} w \right) \, dx \nonumber \\
			\leq & \int_{\Omega} |f'(u)| |A^{\frac{1}{2}} u| \left( |A^{\frac{3}{2}} w_t| + \varepsilon |A^{\frac{3}{2}} w| \right) \, dx \nonumber \\
			\leq & C \|A^{\frac{1}{2}} u\|_{L^{\frac{2n}{n - 2(3 - \delta_0)}}} \left( \|A^{\frac{3}{2}} w_t\|_{L^{\frac{2n}{n - 2(1 - \delta_0)}}} + \varepsilon \|A^{\frac{3}{2}} w\|_{L^{\frac{2n}{n - 2}}} \right) \left( 1 + \|u\|_{\frac{(p-1)n}{4 - 2 \delta_0}}^{p-1} \right) \nonumber \\
			\leq & \frac{k\varepsilon}{4} \|w\|_{V_4}^2 + C.
		\end{align}
		Substituting the estimates from \eqref{3.19-} and \eqref{3.21-} into \eqref{3.23}, and applying Gronwall's inequality, we obtain
		\begin{align}\label{222}
			\|(w, w_t, \zeta)\|_{\mathcal{H}^2} \leq C.
		\end{align}
	Then, the result follows from the above analysis.
	\end{proof}
	
	\begin{remark}\label{re1}{\rm
When the memory term vanishes, we can relax the condition \eqref{assp1} (for the case of $\Gamma\in C^2$) in Theorem \ref{them4.4} and Proposition \ref{lem3.4} to \eqref{assp}
(see also Remark \ref{re2.8}). In fact, we can still obtain the estimate \eqref{5.6} under the weaker condition, by taking $\sigma_1 = 0$ and $\sigma_2 = \min\left\{\alpha, 2 - \alpha\right\}$ in \eqref{4.9}, without the need of using the auxiliary function $J_{\sigma_2}$. The estimate implies $u_{t} \in V_{2+\alpha}$ and $u_{tt} \in V_{\alpha}$ for $(u, u_{t}, \eta) \in \mathcal{A}_{\alpha}$. We multiply both sides of the equation \eqref{1.2} by $A^2 u$ and integrate over $\Omega$, obtaining
		\begin{align}\label{4.28}
			k \|u\|_{V_4}^2 + \int_{\Omega} u_{tt} A^2 u \, dx + \int_{\Omega} A^{\alpha} u_t A^2 u \, dx
			= - \int_{\Omega} (f(u) - h(x)) A^2 u \, dx.
		\end{align}
	Then, by the interpolating inequality  we infer, for $\varepsilon_0 $ small enough,
		\begin{align}\label{4.29}
			&- \int_{\Omega} (f(u) - h(x)) A^2 u \, dx \nonumber \\
			&\le C \int_{\Omega} (1 + |u|^{p-1}) |u| |A^2 u| \, dx + \frac{k}{8} \|u\|_{V_4}^2 + C \|h\|^2 \nonumber \\
			&\le C \|u\|_{L^{\frac{2n}{n - 2(4-\varepsilon_0)}}} \|u\|_{V_4} \left(1 + \|u\|_{L^{\frac{(p-1)n}{4-\varepsilon_0}}}\right) + \frac{k}{8} \|u\|_{V_4}^2 + C \|h\|^2 \nonumber \\
			&\le \frac{k}{4} \|u\|_{V_4}^2 + C,
		\end{align}
		and
		\begin{align}\label{4.30}
			\int_{\Omega} u_{tt} A^2 u \, dx + \int_{\Omega} A^{\alpha} u_t A^2 u \, dx \le \frac{k}{4} \|u\|_{V_4}^2 + C (\|u_{tt}\|^2 + \|u_t\|_{V_{2\alpha}}^2).
		\end{align}
		
	Therefore, substituting \eqref{4.29} and \eqref{4.30} directly into \eqref{4.28} gives $u \in V_4$. Thus, an argument similar to (and simpler than) the regularity analysis for Theorem \ref{them4.4} allows us to conclude that the regularity space of the attractor of the corresponding system is $V_4 \times V_{4-\epsilon}$ ($\epsilon$ can be an arbitrarily small positive constant, and 0 if $\alpha=2$).
		}
	\end{remark}

	\subsubsection{Proof of Theorem \ref{them4.4} for $\Gamma\in C^{\infty}$}
	In contrast to the case of $\Gamma \in C^2$,  the solution to problem \eqref{1.2} in Theorem \ref{thm1} exhibits higher regularity estimate \eqref{esofu} for smooth $\Gamma$, and consequently, the regularity of the solutions to problems \eqref{3.10} and also \eqref{3.11} are enhanced (as shown in the following proposition). Then the desired regularity result is derived from this proposition in combination with a bootstrap argument.
	\begin{proposition}\label{them3.8}
		Let The Assumption \ref{ass1} hold and assume \eqref{assp}. For the solutions of the problems \eqref{3.10} and \eqref{3.11}, we get the following estimates
		\begin{align}\label{3.12p}
			\|(v, v_{t}, \xi)\|_{\mathcal{H}^{\beta}}\le C e^{-\gamma_{7}t},~\mbox{for~any~} \beta>0,
		\end{align}
		with some  $\gamma_{7}>0$, and
		\begin{align}\label{3.12q}
			\|(w,w_{t},\zeta)\|_{\mathcal{H}^{2}}^{2}+\|w_{t}\|_{V_{4-\delta_{0}}}^{2}+\|w_{tt}\|_{2-\delta_{0}}^{2}+\|\zeta_{t}\|_{\mathcal{M}_{2-\delta_{0}}}^{2}\le C.
		\end{align}
	\end{proposition}
	\begin{proof}
		Let $$v_{m}=\sum_{i=1}^{m}d_{im}(t)\omega_{i}(x), \quad \xi_{m}=\sum_{i=1}^{m}(d_{im}(t)-e_{im}(t-s))\omega_{i}(x),$$ $\omega_{i}\in C^{\infty}(\Omega)$ be the Galerkin approximate solutions to problem \eqref{3.10}; that is,
		\begin{align*}
			v_{mtt}+kA^{2}v_{m}+A^{\alpha}v_{mt}+\displaystyle\int_{0}^{\infty}g(s)A^{2}\xi_{m}ds=0,
		\end{align*}
		By virtue of \eqref{P5}--\eqref{P8} for $v_{m}$ (instead of $v$ there) and $\mathcal{F}=0$, we infer, for any $\beta:=\sigma_{1}=\sigma_{2}>0$,
		\begin{align*}
			&\frac{d}{dt}\left(E_{\beta}(v_{m})+\varepsilon I_{\beta}(v_{m})+\frac{\varepsilon}{2}J_{\beta}(v_{m})\right)+\|v_{mt}\|_{V_{\alpha+\beta}}^{2}\\
			&+k\varepsilon\|v_{m}\|_{V_{2+\beta}}^{2}+\frac{\varepsilon}{2}\|\xi_{m}\|_{\mathcal{M}_{2+\beta}}^{2}\le \varepsilon\|v_{mt}\|_{V_{\beta}}^{2}.
		\end{align*}
		Observe $$E_{\beta}(v_{m}(0))+\varepsilon I_{\beta}(v_{m}(0))+\frac{\varepsilon}{2}J_{\beta}(v_{m}(0))\le C$$ by $\omega_{i}\in C^{\infty}$, also note the fact that
		\begin{align*}
			E_{\beta}(v_{m})+\varepsilon I_{\beta}(v_{m})+\frac{\varepsilon}{2}J_{\beta}(v_{m})\sim \|v_{m}\|_{V_{2+\beta}}^{2}+\|v_{mt}\|_{V_{\beta}}^{2}+\|\xi_{m}\|_{\mathcal{M}_{2+\beta}}^{2}.
		\end{align*}
Hence, taking $\varepsilon$ small enough gives (by the Gronwall inequality)
		\begin{align*}
			\|(v_{m}, v_{mt}, \xi_{m})\|_{\mathcal{H}^{\beta}}^{2}\le C(\|(v_{m}(0),v_{mt}(0),\xi_{m}(0)\|_{\mathcal{H}^{\beta}})e^{-\gamma_{7}t},
		\end{align*}
with some $\gamma_{7}>0$. Thus, utilizing the lower semicontinuity of the norm associated with the weak limit, along with the uniqueness of the weak solution, we obtain the estimate \eqref{3.12p}.
		
		For the equation \eqref{3.15+}, replacing $\varphi$ by $q$, and taking $\mathcal{F}=f^{\prime}(u)u_{t}$, $\sigma_{1}=\alpha$ and $\sigma_{2}=\min\{2\alpha,2\}$ in \eqref{P5}--\eqref{P8} yields
		\begin{align}\label{3.16+}
			&	\frac{d}{dt}\left(E_{\alpha}(q)+\varepsilon I_{\sigma_{2}}(q)+\frac{\varepsilon}{2}J_{\sigma_{2}}(q)\right)+\|q_{t}\|_{V_{2\alpha}}^{2}+\frac{\varepsilon}{2}\|\zeta_{t}\|_{\mathcal{M}_{2+\sigma_{2}}}^{2}+k\varepsilon\|q\|_{V_{2+\sigma_{2}}}^{2}\nonumber\\
			\le&\varepsilon\|q_{t}\|_{V_{\sigma_{2}}}^{2}-\int_{\Omega}f^{\prime}(u)u_{t}(A^{\alpha}q_{t}+\varepsilon A^{\sigma_{2}}q)dx.
		\end{align}
		By the interpolating inequality and from \eqref{assp} and \eqref{3.12p}, we have,  for $0<\varepsilon_{0}<<1$,
		\begin{align}\label{3.17+}
			&	-\int_{\Omega}f^{\prime}(u)u_{t}(A^{\alpha}q_{t}+\varepsilon A^{\sigma_{2}}q)dx\nonumber\\
			\le & C\int_{\Omega}(1+|u|^{p-1})(|v_{t}|+|q|)(|A^{\alpha}q_{t}|+\varepsilon |A^{\sigma_{2}}q|)dx\nonumber\\
			\le&C\left(\|v_{t}\|_{L^{\frac{2n}{n-2(2+\sigma_{2}-\varepsilon_{0})}}}+ \|q\|_{L^{\frac{2n}{n-2(2+\sigma_{2}-\varepsilon_{0})}}}\right)\left(\|A^{\alpha}q_{t}\|+\|A^{\sigma_{2}}q\|_{L^{\frac{2n}{n-2(2-\sigma_{2})}}}\right)\nonumber\\
			&\times\left(1+\|u\|_{L^{\frac{(p-1)n}{2+\sigma_{2}-\varepsilon_{0}}}}^{p-1}+\|u\|_{L^{\frac{(p-1)n}{4-\varepsilon_{0}}}}^{p-1}\right)\nonumber\\
			\le&\frac{k\varepsilon}{2}\|q\|_{V_{2+\sigma_{2}}}^{2}+\frac{1}{2}\|q_{t}\|_{V_{2\alpha}}^{2}+C.
		\end{align}
		Insert \eqref{3.17+} into \eqref{3.16+} and take $\varepsilon$ small enough; it follows that
		\begin{align*}
			&	\frac{d}{dt}\left(E_{\alpha}(q)+\varepsilon I_{\sigma_{2}}(q)+\frac{\varepsilon}{2}J_{\sigma_{2}}(q)\right)+\frac{1}{4}\|q_{t}\|_{V_{2\alpha}}^{2}+\frac{\varepsilon}{2}\|\zeta_{t}\|_{\mathcal{M}_{2+\sigma_{2}}}^{2}+\frac{k\varepsilon}{2}\|q\|_{V_{2+\sigma_{2}}}^{2}\le C.
		\end{align*}
Therefore,
		\begin{equation}\label{3.44}
			\|w_{t}\|_{V_{2+\alpha}}^{2}+\|w_{tt}\|_{V_{\alpha}}^{2}+\|\zeta_{t}\|_{\mathcal{M}_{2+\alpha}}^{2}\le C,
		\end{equation}
owing to the estimate \eqref{P9} and $q=w_{t}$.

		 Replacing $\varphi$ with $w$ and taking $\sigma_{1}=\frac{2+\alpha}{2}, \sigma_{2}=2$ in \eqref{P5}-- \eqref{P8} leads to
		\begin{align}\label{3.37}
			&\frac{d}{dt}\left(E_{\frac{2+\alpha}{2}}(w)+\varepsilon I_{2}(w)+\frac{\varepsilon}{2}J_{2}(w)\right)+\|w_{t}\|_{V_{\frac{2+\alpha}{2}+\alpha}}^{2}+\frac{\varepsilon}{2}\|\zeta\|_{\mathcal{M}_{4}}^{2}+k\varepsilon\|w\|_{V_{4}}^{2}\nonumber\\
			\le&\varepsilon\|w_{t}\|_{V_{2}}^{2}-\int_{\Omega}f(u)\left(A^\frac{2+\alpha}{2}w_{t}+\varepsilon A^{2}w\right)dx+\int_{\Omega}h(x)\left(A^{\frac{2+\alpha}{2}}w_{t}+\varepsilon A^{2}w\right)dx.
		\end{align}
		We observe
		\begin{align}\label{3.38}
			\int_{\Omega}h(x)\left(A^{\frac{2+\alpha}{2}}w_{t}+\varepsilon A^{2}w\right)dx\le \frac{1}{4}\|w_{t}\|_{V_{2+\alpha}}^{2}+\frac{k\varepsilon}{4}\|w\|_{V_{4}}^{2}+C\|h\|^{2},
		\end{align}
		and for $0<\varepsilon_{0}<<1$,
		\begin{align}\label{3.39}
			&-\int_{\Omega}f(u)\left(A^\frac{2+\alpha}{2}w_{t}+\varepsilon A^{2}w\right)dx\nonumber\\
			\le& C\int_{\Omega}(1+|u|^{p-1})(|v|+|w|)\left(\left|A^{\frac{2+\alpha}{2}}w_{t}\right|+\varepsilon |A^{2}w|\right)dx\nonumber\\
			\le&C\left(\|v\|_{L^{\frac{2n}{n-2(4-\varepsilon_{0})}}}+\|w\|_{L^{\frac{2n}{n-2(4-\varepsilon_{0})}}}\right)\left(\|w_{t}\|_{V_{2+\alpha}}+\varepsilon \|w\|_{V_{4}}\right)\left(1+\|u\|_{L^{\frac{(p-1)n}{4-\varepsilon_{0}}}}^{p-1}\right)\nonumber\\
			\le &\frac{1}{4}\|w_{t}\|_{V_{2+\alpha}}^{2}+\frac{k\varepsilon}{4}\|w\|_{V_{4}}^{2}+C,
		\end{align}
 by \eqref{3.12p} and the interpolation inequality.	Inserting \eqref{3.38} and \eqref{3.39} into \eqref{3.37},  we obtain
		\begin{align}\label{3.40}
			&\frac{d}{dt}\left(E_{\frac{2+\alpha}{2}}(w)+\varepsilon I_{2}(w)+\frac{\varepsilon}{2}J_{2}(w)\right)+\frac{1}{2}\|w_{t}\|_{V_{\frac{2+\alpha}{2}+\alpha}}^{2}+\frac{\varepsilon}{2}\|\zeta\|_{\mathcal{M}_{4}}^{2}+\frac{k\varepsilon}{2}\|w\|_{V_{4}}^{2}\le C,
		\end{align}
		and for $\varepsilon$ small enough, combining \eqref{P9} and \eqref{3.40} gives
		\begin{align}\label{3.42}
			\|(w,w_{t},\zeta)\|_{\mathcal{H}^{\frac{2+\alpha}{2}}}^{2}\le C.
		\end{align}
		Combining this together with \eqref{3.12p}, \eqref{3.44} and \eqref{3.42}, we see that the attractor $\mathcal{A}_{\alpha}$ is bounded in $V_{\frac{6+\alpha}{2}}\times V_{2+\alpha}\times \mathcal{M}_{\frac{6+\alpha}{2}}$ by Remark A.5 in \cite{Conti-2005-}.

Next, we exploit the bootstrap argument again.
 Let $$\rho_{1}^{*} := \sigma_{1} \in \left(\frac{3\alpha+2}{4},\frac{\alpha + 2}{2}\right], ~~\sigma_{2} = \frac{2+\alpha}{2}$$ in \eqref{4.9p} and re-estimate inequality \eqref{4.27-}. Then we have, for any $0<\varepsilon_{0}<<1$,
		\begin{align}\label{4.27+}
			& -\int_{\Omega}\left(f(u^{1})-f(u^{2})\right)\left(A^{\rho_{1}^{*}}\phi_{t}+\varepsilon A^{\frac{2+\alpha}{2}}\phi\right)dx\nonumber\\
			\le& C\left(\|\phi\|_{L^{\frac{2n}{n-2\left(2+\frac{2+\alpha}{2}-\varepsilon_{0}\right)}}}+\|v_{1}-v_{2}\|_{L^{\frac{2n}{n-2\left(2+\frac{2+\alpha}{2}-\varepsilon_{0}\right)}}}\right)\nonumber\\
			&\times \left(\left\|A^{\rho_{1}^{*}}\phi_{t}\right\|_{L^{\frac{2n}{n-2(2+\alpha-2\rho_{1}^{*})}}}+\varepsilon \left\|A^{\frac{2+\alpha}{2}}\phi\right\|_{\frac{2n}{n-2\left(2-\frac{\alpha+2}{2}\right)}}\right)\nonumber\\
			&\times\left(1+\|u^{1}\|_{L^{\frac{(p-1)n}{4+\frac{2+\alpha}{2}+\alpha-2\rho_{1}^{*}-\varepsilon_{0}}}}^{p-1}+\|u^{2}\|_{L^{\frac{(p-1)n}{4+\frac{2+\alpha}{2}+\alpha-2\rho_{1}^{*}-\varepsilon_{0}}}}^{p-1}\right)\nonumber\\
			\le&\frac{k\varepsilon}{2}\|\phi\|_{V_{\frac{6+\alpha}{2}}}^{2}+C\left(\|v_{1}\|_{V_{\frac{6+\alpha}{2}}}^{2}+\|v_{2}
\|_{V_{\frac{6+\alpha}{2}}}^{2}+\|v_{1}\|^{\mu_{1}}+\|v_{2}\|^{\mu_{1}}\right)+C\|\phi\|^{\mu_{1}}.
		\end{align}
		Here, when $4<n\le 8+4\min\{\alpha,1\}$ and $1\le p<p^{*}$, one has the embedding $V_{\frac{6+\alpha}{2}}\hookrightarrow L^{\frac{(p-1)n}{2+\frac{2+\alpha}{2}}}$. Consequently, by selecting $\rho_{1}^{*}=\frac{2+\alpha}{2}$, the inequality \eqref{4.27+} is ensured to hold. When $n>8+4\min\{\alpha,1\}$ and $1\le p<p^{*}$, the validity of the last inequality in \eqref{4.27+} hinges on the following relationship between the exponent $p$ and the dimension $n$:
		\begin{align*}
			\frac{(p-1)n}{4+\frac{2+\alpha}{2}+\alpha-2\rho_{1}^{*}}< \frac{2n}{n-2\left(2+\frac{2+\alpha}{2}\right)} \Longrightarrow p<p^{*}\le \frac{n+4+2\alpha-4\rho_{1}^{*}}{n-2\left(2+\frac{2+\alpha}{2}\right)},
		\end{align*}
	where we set
		\begin{align*}
			\rho_{1}^{*}=\min \left\{\frac{2+\alpha}{2}, \left(\frac{3}{4}+\frac{1+\min\{\alpha,1\}}{n-4}\right)\alpha+\frac{3-2\min\{\alpha,1\}}{2}+\frac{2(1+\min\{\alpha,1\})}{n-4}\right\}.
		\end{align*}
		Inserting \eqref{4.27+} into \eqref{4.9p} and using \eqref{3.12p} shows
		\begin{align*}
			\|W_{1}-W_{2}\|_{\mathcal{H}^{\rho_{1}^{*}}}^{2}\le Ce^{-\gamma_{8} t}+C\left(1-e^{-\gamma_{8}t}\right)\sup_{s\in[0,t]}\|w^{1}-w^{2}\|^{\mu_{1}},
		\end{align*}
 with some$\gamma_{8}>0$. Hence, from Lemma \ref{lemma4.3} we also have $w_{t}\in V_{2+\rho_{1}^{*}}$. Then, replacing $\sigma_{1}=\frac{2+\alpha}{2}$ with $\sigma_{1}=\frac{2+\rho_{1}^{*}}{2}$ in \eqref{3.37} gives $w\in V_{2+\frac{2+\rho_{1}^{*}}{2}}$. Continuing in this way, we can obtain an increasing sequence $\{\rho_{m}^{*}\}_{m\ge0}$ with $\rho_{0}^{*}=\alpha$, such that $w_{t}\in V_{2+\rho_{m}^{*}}$ and $w\in V_{2+\frac{2+\rho_{m}^{*}}{2}}$. Explicitly, this sequence  is given by
  $$\rho_{m+1}^{*}=\frac{2+\rho_{m}^{*}}{2} \quad\quad \mbox{if } \ 4<n\le 8+4\min\{\alpha,1\},$$ and $ \rho_{m+1}^{*}$ is equal to
		\begin{align*}
			\min\left\{\frac{2+\rho_{m}^{*}}{2}, \left(\frac{3}{4}+\frac{1+\min\{\alpha,1\}}{n-4}\right)\rho_{m}^{*}+\frac{3-2\min\{\alpha,1\}}{2}+\frac{2(1+\min\{\alpha,1\})}
{n-4}\right\}
		\end{align*}
if $n>8+4\min\{\alpha,1\}$.
Hence, we deduce
		\begin{align*}
			\lim_{m\to\infty}\rho_{m}^{*}=2,~~\mbox{if}~ 4<n\le 8+4\min\{\alpha,1\}.
		\end{align*}
		For $n> 8+4\min\{\alpha,1\}$, denote by $\rho^{*}$ the limit value of $\rho_{m}^{*}$ as $m\rightarrow \infty$. Then, it is clear that
\begin{align*}
		\rho^{*}=	\min\left\{\frac{2+\rho^{*}}{2}, \left(\frac{3}{4}+\frac{1+\min\{\alpha,1\}}{n-4}\right)\rho^{*}+\frac{3-2\min\{\alpha,1\}}{2}+\frac{2(1+\min\{\alpha,1\})}{n-4}\right\},
		\end{align*}
		which indicates
		\begin{align*}
			\lim_{m\to \infty}\rho_{m}^{*}=\min\left\{2, 6-4\min\{\alpha,1\}+\frac{8(1+\min\{\alpha,1\})(4-2\min\{\alpha,1\})}{n-8-4\min\{\alpha,1\}}\right\}=2.
		\end{align*}
		For the case of $1\le n\le4$ (and so $1\le p< \infty$), we can also produce a sequence $\{\widehat{\rho}_{m}\}_{m\ge 0}$ with $\widehat{\rho}_{0}=\alpha$, satisfying $\widehat{\rho}_{m}=1+\frac{\widehat{\rho}_{m-1}}{2}\to 2$ as $m\to \infty$.

Finally, with some arguments similar to those for \eqref{222}, we infer that $(w,w_{t}, \zeta)\in \mathcal{H}^{2}$. In conclusion, we justify the estimate \eqref{3.12q}.
	\end{proof}

\vspace{0.5cm}

{\bf Data Availability}

Data sharing is not applicable to this article as no datasets were generated or analyzed during the current study.

{\bf Conflict of interest}

The authors have no competing interests to declare.

{\bf Ethical Statement}

There are no ethical concerns applicable to our research.


\vspace{0.4cm}	

	

\begin{thebibliography}{99}
	\bibitem{Araruna-2018} F.D. Araruna,  P. Braz e Silva,  P. Queiroz-Souza,
	Asymptotic limits and stabilization for the 2D nonlinear Mindlin-Timoshenko system,
	Anal. PDE. 11 (2018) 351-382.
	
	\bibitem{Azevedo-2023}V.T. Azevedo, E. Bonotto, A.C. Cunha, M.J.D. Nascimento, Existence and stability of pullback exponential attractors for a nonautonomous semilinear evolution equation of second order, J. Differ. Equ. 365 (2023) 521-559.


	\bibitem{Berger-1955} H.M. Berger, A new approach to the analysis of large deflections of plates, J. Appl. Mech. 22 (1955) 465-472.
	\bibitem{Cavalcanti-2016} M. Cavalcanti, L. Fatori, T.F. Ma, Attractors for wave equations with degenerate memory, J. Differ. Equ. 260 (2016) 56-83.

\bibitem{Chen-1982}G. Chen, D.L. Russell, A mathematical model for linear elastic systems with structural damping, Q. Appl. Math. 39 (1982) 433-454.

	\bibitem{Chueshov-2015}I. Chueshov, Dynamics of Quasi-Stable Dissipative Systems, Cham: Springer, 2015.

		\bibitem{Chueshov-2013} I. Chueshov, I. Lasiecka,  Well-posedness and long-time behavior in nonlinear dissipative hyperbolic-like evolutions with critical exponents, HCDTE Lecture Notes,
	Part I, Nonlinear Hyperbolic PDEs, Dispersive and Transport Equations, AIMS
	Ser. Appl. Math., 6, Am. Inst. Math. Sci. (AIMS), Springfield, MO, 2013.
	
	\bibitem{Conti-2005-} M. Conti, V. Pata, Weakly dissipative semilinear equations of viscoelasticity, Commun. Pure Appl. Anal. 4(2005) 705-720.
	
	
	\bibitem{Dafermos-1970}C.M. Dafermos, Asymptotic stability in viscoelasticity, Arch. Ration. Mech. Anal. 37 (1970) 297-308.
	
	\bibitem{Di Plinio-2008}F. Di Plinio, V. Pata,  S. Zelik, On the strongly damped wave equation with memory, Indiana Univ. Math. J. 57 (2008) 757-780.
	
\bibitem{Duan25}Y.Y. Duan, Long-term behavior of solutions to elastic dynamical systems with memory effects, Ph.D. Dissertation, Fudan Uni., Shanghai, March, 2025.

	\bibitem{Freitas-2025}M.M. Freitas, D.S. Almeida J\'unior,  A.J.A. Ramos, M.J. Dos Santos, R.Q. Caljaro,
	Long-time dynamics and singular limit of a shear beam model,
	Math. Ann. 391 (2025) 2149-2171.
	\bibitem{Freitas-2018}M.M. Freitas, P. Kalita, J.A. Langa, Continuity of non-autonomous attractors for hyperbolic perturbation of parabolic equations, J. Differ. Equ. 264 (2018) 1886-1945.
	
	
	\bibitem{Gatti-2008}  S. Gatti, A. Miranville, V. Pata, S. Zelik, Attractors for semi-linear equations of viscoelasticity with
	very low dissipation, Rocky Mount. J. Math. 38 (2008) 1117-1138.
	
	\bibitem{Giorgi-2009} C. Giorgi, M.G. Naso, V. Pata, M. Potomkin, Global attractors for the extensible thermoelastic beam system, J. Differ. Equ. 246 (2009) 3496-3517.
	
	\bibitem{131}C.	Giorgi, V. Pata, A. Marzocchi,
	Asymptotic behavior of a semilinear problem in heat conduction with memory,
	NoDEA Nonlinear Differential Equations Appl. 5 (1998) 333-354.
	
\bibitem{Gomes-2024} H. Gomes Tavares, M.A. Jorge Silva, I. Lasiecka, V. Narciso, Dynamics of extensible beams with nonlinear noncompact energy-level damping, Math. Ann. 390 (2024) 1821-1862.	

\bibitem{IX24} Y. Inahama, Y. Xu, X.Y. Yang, Moderate deviations for rough differential equations, Bull. London Math. Soc, 56 (2024) 2738-2748.


\bibitem{Silva-2015} M.A. Jorge Silva, V. Narciso, Attractors and their properties for a class of nonlocal extensible beams, Discrete Contin. Dyn. Syst. 35 (2015) 985-1008.
	

	\bibitem{la26} I. Lasiecka, J.H. Rodrigues, M. Roy, Attractors for second order in time non-conservative dynamics with nonlinear damping, J. Differ. Equ. 451 (2026) 113760.

\bibitem{LX25}C. Li, H.-K. Xu, Dynamic behaviors for the acoustic model with variable coefficients and nonautonomous damping, Z. Angew. Math. Phys. 76 (2025)  article number 17.


	\bibitem{Li-2020}Y.A. Li, Z.J. Yang, Optimal attractors of the Kirchhoff wave model with structural nonlinear damping,  J. Differ. Equ. 268 (2020) 7741-7773.
	\bibitem{Li-2021}Y.A. Li, Z.J. Yang,
	Strong attractors and their continuity for the semilinear wave equations with fractional damping, Adv. Differential Equations, 26 (2021) 45-82.

\bibitem{LX20}Q. Liu, Y. Xu, J. Kurths, Bistability and stochastic jumps in an airfoil system with viscoelastic
material property and random fluctuations, Commun. Nonlinear Sci. Numer. Simulat.  84 (2020)  105184.


\bibitem{Liu-2024}Z.M. Liu, Z.J. Yang, Y.Y. Guo,
	Stability of strong attractors for the extensible beam equation with gentle dissipation, J. Math. Anal. Appl. 533 (2024) 127999.
	
	\bibitem{Ma-2010}T.F. Ma, V. Narciso, Global attractor for a model of extensible beam with nonlinear damping and source terms, Nonlinear Anal. 73 (2010) 3402-3412.

	\bibitem{Ma-2004}Q.Z. Ma, C.K. Zhong, Existence of strong global attractors for hyperbolic equation with linear memory, Appl. Math. Comput. 157 (2004) 745-758.
	
	\bibitem{Messaoudi-2017} S.A. Messaoudi, A. Bonfoh, S.E. Mukiawa, C.D. Enyi, The global attractor for a suspension bridge with memory and partially hinged boundary conditions, ZAMM Z. Angew. Math. Mech. 97 (2017) 159-172.
	

\bibitem{Miranda-2025}L.G.R. Miranda, C.A. Raposo, M.M. Freitas,
	Global and exponential attractors for a suspension bridge model with nonlinear damping,
	J. Differ. Equ. 431 (2025) 113217.
	
	
	\bibitem{Pata-2001}V. Pata, A. Zucchi, Attractors for a damped hyperbolic equation with linear memory, Adv. Math. Sci. Appl. 11 (2001) 505-529.
	
\bibitem{PSX24}B. Pei, B. Schmalfuss, Y. Xu, Almost sure averaging for evolution equations driven by fractional
brownian motions, SIAM J. Appl. Dyn. Sys. 23 (2024) 2807-2852.

	\bibitem{Temam-1988} R. Temam, Infinite-dimensional Dynamical Systems in Mechanics and Physics, Springer
	Verlag, New York, 1988.
	

	\bibitem{Yang-2013} Z.J. Yang, On an extensible beam equation with nonlinear damping and source terms, J. Differ. Equ.
	254 (2013) 3903-3927.
	\bibitem{Yang-2017}Z.J. Yang, P.Y. Ding,
	Longtime dynamics of Boussinesq type equations with fractional damping, Nonlinear Anal. 161 (2017) 108-130.
	\bibitem{Yang-2018}Z.J. Yang, Z.M. Liu,
	Stability of exponential attractors for a family of semilinear wave equations with gentle dissipation, J. Differ. Equ. 264 (2018) 3976-4005.
	
	\bibitem{Yang-2024} B. Yang, Y.M. Qin, A. Miranville,
	K. Wang, Existence and regularity of global attractors for a Kirchhoff wave equation with strong damping and memory, Nonlinear Anal. Real World Appl. 79 (2024) 104096.
	
	
	
	
\end{thebibliography}
\end{document}